\documentclass[a4paper]{llncs}

\usepackage[breaklinks,colorlinks,citecolor=blue,linkcolor=blue,urlcolor=blue]{hyperref}
\usepackage[utf8x]{inputenc}
\usepackage{amsmath}
\usepackage{graphicx}
\usepackage{bm}
\usepackage{upgreek}
\usepackage{array}
\usepackage{amssymb} 
\usepackage{bbm}
\usepackage{dsfont}
\usepackage{mathtools}
\usepackage{enumerate}
\usepackage{paralist}
\usepackage{cleveref}
\usepackage{tabularx}
\usepackage{booktabs}
\usepackage{nicefrac}
\usepackage{svg}

\newcommand{\ubar}[1]{\text{\b{$#1$}}}

\usepackage{tikz} 
\usetikzlibrary{shapes.geometric}
\usetikzlibrary{positioning}
\tikzstyle{Lvertex}=[circle,draw, minimum size=1em, inner sep = 1pt]
\tikzstyle{Svertex}=[circle,draw, minimum size=4pt, inner sep = 0]
\tikzstyle{Svertexblack}=[circle,fill, minimum size=4pt, inner sep = 0]
\tikzstyle{arc}=[-latex, thick]
\tikzstyle{edge}=[thick]

\pgfdeclarelayer{bg}
\pgfsetlayers{bg,main}

\usepackage[ruled,linesnumbered,noend,noline]{algorithm2e}

\SetCommentSty{mycommfont}
\DontPrintSemicolon

\renewcommand{\vec}[1]{#1}

\newcommand{\R}{\mathds{R}}

\newcommand{\Qp}{\mathds{Q}_{\geq 0}}
\newcommand{\Qpp}{\mathds{Q}_{>0}}

\newcommand{\N}{\mathds{N}}
\newcommand{\E}{\mathds{E}}
\renewcommand{\Pr}{\mathds{P}}
\newcommand{\assign}{\leftarrow}

\newcommand{\NP}{\mathsf{NP}}

\newcommand{\Z}{\mathds{Z}}

\newcommand{\eps}{\varepsilon}
\newcommand{\abs}[1]{\lvert{#1}\rvert}
\newcommand{\card}[1]{\lvert{#1}\rvert}
\newcommand{\norm}[1]{\lVert{#1}\rVert}
\newcommand{\floor}[1]{\lfloor{#1}\rfloor}

\newcommand{\pd}[2]{\frac{\partial {#1}}{\partial{#2}}}
\newcommand{\define}{\coloneqq}

\newcommand{\ones}{\mathbf{1}}
\newcommand{\ind}{\chi}

\DeclareMathOperator*{\argmax}{argmax}

\DeclareMathOperator{\diag}{diag}
\DeclareMathOperator{\ber}{Ber}

\title{Packing under Convex Quadratic Constraints
\thanks{We acknowledge funding through the DFG CRC/TRR~154, Subproject~A007.}}

\author{Max Klimm\inst{1} \and Marc E. Pfetsch\inst{2}  \and \mbox{Rico Raber}\inst{3} \and Martin Skutella\inst{3}}
\institute{School of Business and Economics, HU Berlin, Spandauer Str.~1, 10178~Berlin, Germany,  \email{max.klimm@hu-berlin.de}. \and
  Department of Mathematics, TU Darmstadt, Dolivostr.~15, 64293 Darmstadt, Germany, \email{pfetsch@mathematik.tu-darmstadt.de} \and
  Institute of Mathematics, TU Berlin, Stra\ss e des 17.~Juni 136, 10623~Berlin, Germany, \email{$\{$raber,skutella$\}$@math.tu-berlin.de}}
 
\authorrunning{M. Klimm, M. E. Pfetsch, R. Raber, and M. Skutella}

\date{\today}

\begin{document}
\maketitle

\begin{abstract} 
We consider a general class of binary packing problems with a convex quadratic knapsack constraint.
We prove that these problems are APX-hard to approximate and present constant-factor approximation
algorithms based upon three different algorithmic techniques: (1) a rounding technique tailored
to a convex relaxation in conjunction with a non-convex relaxation whose approximation ratio equals the
golden ratio; (2) a greedy strategy; (3) a randomized rounding method
leading to an approximation algorithm for the more general case with multiple convex quadratic
constraints. We further show that a combination of the first two strategies can be used to yield a monotone algorithm leading to a strategyproof mechanism for a game-theoretic variant of the problem. Finally, we present a computational study of the empirical approximation of the three
algorithms for problem instances arising in the context of real-world gas transport networks.
\end{abstract}

\section{Introduction}

We consider packing problems with a convex quadratic knapsack constraint of the form
\begin{align}
\begin{split}
\text{maximize} \qquad &\vec p^\top \vec x\\
\text{subject to} \qquad &\vec x^\top \vec W \vec x \leq c,\\
&\vec x \in \{0,1\}^n, \label{eq:problem}
\end{split}\tag{$P$}
\end{align}
where $\vec W \in \Qp^{n \times n}$ is a symmetric positive semi-definite (psd) matrix with non-negative entries, $\vec p \in \Qp^n$ is a non-negative profit vector, 
and $c \in \Qp$ is a non-negative budget. 
Such convex and quadratically constrained packing problems are clearly $\NP$-complete since they contain the classical (linearly constrained) $\NP$-complete knapsack problem \cite{karp1972} as a special case when $\vec W$ is a diagonal matrix. In this paper, we therefore focus on the development of approximation algorithms. For some $\rho \in [0,1]$, an algorithm is a $\rho$-approximation algorithm if its runtime is polynomial in the input size and for every instance, it computes a solution with objective value at least $\rho$ times that of an optimum solution.
The value $\rho$ is then called the \emph{approximation ratio} of the algorithm. We note that the assumption on $W$ being psd is necessary in order to allow for sensible approximation. To see this, observe that when $W$ is the adjacency matrix of an undirected graph and $c = 0$, \eqref{eq:problem} encodes the problem of finding an independent set of maximal weight, which is $\mathsf{NP}$-hard to approximate within a factor better than $n^{-(1-\epsilon)}$ for any $\epsilon > 0$, even in the unweighted case \cite{Hastad1999}. 

The packing problems that we consider also have a natural interpretation in terms of mechanism design. 
Consider a situation where a set of $n$ selfish agents demands a service, and the subsets of agents that can be served simultaneously are modeled by a convex quadratic packing constraint. Each agent~$j$ has private information $p_j$ about its willingness to pay for receiving the service.
In this context, a (direct revelation) mechanism takes as input the matrix $\vec W$ and the budget $c$. It then elicits the private value $p_j$ from agent~$j$. Each agent~$j$ may misreport a value $p'_j$ instead of their true value $p_j$ if this is to their benefit. The mechanism then computes a solution $\vec x \in \{0,1\}^n$ to \eqref{eq:problem} as well as a payment vector $\vec g \in \Qp^n$. A mechanism is strategyproof if no agent has an interest in misreporting $p_j$, no matter what the other agents report.

Before we present our results on approximation ratios and mechanisms for non-negative, convex, and quadratically constrained packing problems, we give two real-world examples that fall into this category.

\begin{example}[Welfare maximization in gas supply networks]\label{ex:gas}
Consider a gas pipe\-line modeled by a directed graph $G = (V,E)$ with different entry and exit nodes. There is a set of $n$ transportation requests $(s_j,t_j,q_j,p_j)$, $j \in [n] \define \{1, \dots, n\}$, each specifying an entry node $s_j \in V$, an exit node $t_j \in V$, the amount of gas to be transported $q_j \in \Qp$, and an economic value $p_j \in \Qp$. One model for gas flows in pipe networks is given by the Weymouth equations \cite{weymouth1912problems} of the form
\begin{align*}
\beta_e\, q_e\, \abs{q_e} = \pi_u - \pi_v \quad \text{ for all $e = (u,v) \in E$.}	
\end{align*}
Here, the parameter $\beta_e \in \Qpp$ is a pipe specific value that depends on physical properties of the pipe segment modeled by the edge, such as length, diameter, and roughness. Positive flow values $q_e > 0$ denote flow from~$u$ to~$v$, while a negative $q_e$ indicates flow in the opposite direction. The value $\pi_u$ denotes the square of the pressure at node $u \in V$. In real-life gas networks, there is typically a bound $c \in \Qp$ on the maximal difference of the squared pressures in the network. For the operation of gas networks, it is a natural problem to find the welfare-maximal subset of transportation requests that can be satisfied simultaneously while satisfying the pressure constraint.

To illustrate this problem, we consider the particular case in which the network has a path topology similar to the one depicted in Figure~\ref{fig:gas_network}.  We assume that for each request the entry node is left of the exit node. Thus, the pressure in the pipe is decreasing from left to right. For $j \in [n]$, let $E_j \subseteq E$ denote the set of edges on the unique $(s_j,t_j)$-path in $G$. Indexing the vertices $v_0,\dots,v_k$ and edges $e_1, \dots, e_k$ from left to right, the maximal squared pressure difference in the pipe is given by
\begin{align*}
\pi_0 - \pi_k &= \sum_{i=1}^k \bigl( \pi_{i-1} - \pi_i \bigr)
= \sum_{i=1}^{k}\beta_{e_i} \Bigl(\sum_{j \in [n] : e_i \in E_j} q_j\, x_j\Bigr)^{\!2},
\end{align*}
where~$x_j\in\{0,1\}$ indicates whether transportation request $j \in [n]$ is being served.
For the matrix $\vec W = (w_{ij})_{i,j \in [n]}$ defined by $w_{ij} = \sum_{e \in E_i \cap E_j} \beta_e\, q_i\,q_j$, the pressure constraint can be formulated as $\vec x ^{\top} \vec W \vec x \leq c$. To see that the matrix~$\vec W$ is positive semi-definite, we write $\vec W = \sum_{e \in E} \beta_e\, \vec q^e\, (\vec q^e)^{\top}$, where $\vec q^e \in \Qp^n$ is defined as $q^e_i = q_i$ if $e \in E_i$, and $q^e_i = 0$, otherwise.

Gas networks are particularly interesting from a mechanism design perspective, since several countries employ or plan to employ auctions to allocate gas network capacities \cite{newbery2002}, but theoretical and experimental work uses only linear flow models \cite{mccabe1989,rassenti1994}, thus ignoring the physics of the gas flow.
\end{example}

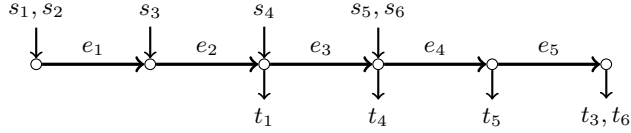
\begin{figure}[tb]
\centering
\begin{tikzpicture}[xscale=1.5]
\node[Svertex] (a) at (0,0) {};
\node (a+) at (0,0.7) {$s_1, s_2$};
\node[Svertex] (b) at (1,0) {};
\node (b+) at (1,0.7) {$s_3$};
\node[Svertex] (c) at (2,0) {};
\node (c+) at (2,0.7) {$s_4$};
\node (c-) at (2,-0.7) {$t_1$};
\node[Svertex](d) at (3,0) {};
\node (d+) at (3,0.7) {$s_5,s_6$};
\node (d-) at (3,-0.7) {$t_4$};
\node[Svertex] (e) at (4,0) {};
\node (e-) at (4,-0.7) {$t_5$};
\node[Svertex] (f) at (5,0) {};
\node (f-) at (5,-0.7) {$t_3,t_6$};
\begin{footnotesize}
\draw[edge,->] (a+) -- (a);
\draw[edge,->] (b+) -- (b);
\draw[edge,->] (c+) -- (c);
\draw[edge,->] (c) -- (c-);
\draw[edge,->] (d+) -- (d);
\draw[edge,->] (e) -- (e-);
\draw[edge,->] (d) -- (d-);
\draw[edge,->] (f) -- (f-);
\draw[very thick,->] (a) --node[above]{$e_1$} (b);
\draw[very thick,->] (b) --node[above]{$e_2$} (c);
\draw[very thick,->] (c) --node[above]{$e_3$} (d);
\draw[very thick,->] (d) --node[above]{$e_4$} (e);
\draw[very thick,->] (e) --node[above]{$e_5$} (f);
\end{footnotesize}
\end{tikzpicture}
\caption{\label{fig:gas_network}
Gas network with feed-in and feed-out nodes.
}
\end{figure}

\begin{example}[Processor speed scaling]
Consider a mobile device with battery capacity $c$ and $k$ compute cores. Further, there is a set of $n$ tasks $(\vec q^j, p_j)$, each specifying a load $\vec q^j \in \Qp^k$ for the $k$ cores and a profit~$p_j$. The computations start at time $0$ and all computations have to be finished at time $1$. In order to adapt to varying workloads, the compute cores can run at different speeds. In the speed scaling literature, it is a common assumption that energy consumption of core~$i$ when running at speed $s$ is equal to $\beta_i\, s^2$, where $\beta_i \in \Qpp$ is a core-specific parameter \cite{bansal2004,irani2005,Wierman2012}.\footnote{Other works assume that the relationship is cubic, but experiments conducted by Wierman et al.~\cite{Wierman2012} suggest that the relationship is closer to quadratic than cubic.} The goal is to select a profit-maximal subset of tasks that can be scheduled in the available time with the available battery capacity. Given a subset of tasks, it is without loss of generality to assume that each core runs at the minimal speed such that the core finishes at time $1$, i.e., every core~$i$ runs at speed $\sum_{j \in [n]} x_j\, q^j_i$ so that the total energy consumption is $\smash{\sum_{i=1}^k \beta_{i} (\sum_{j \in [n]} x_j\, q^j_i)^2}$. The energy constraint can thus be formulated as a convex quadratic constraint.

Mechanism design problems for processor speed scaling are interesting when the tasks are controlled by selfish agents and access to computation on the energy-constrained device is determined via an auction. 
\end{example}

\subsection{Our Results}

In Section~\ref{cf_approx} we derive a $\phi$-approximation algorithm for packing problems with convex quadratic constraints 
 where $\phi = (\sqrt{5}-1)/2 \approx 0.618$ is the inverse golden ratio.
The algorithm first solves a convex relaxation and scales the solution by $\phi$, which turns it 
into a feasible solution to a second non-convex relaxation. 
The latter relaxation has the property that any solution
can be transformed into a solution with at most one fractional component without decreasing the objective value.
In the end, the algorithm returns the integral part of the transformed solution. 
Combining this procedure with a partial enumeration scheme yields a $\phi$-approximation; 
see \Cref{thm:phi_approximation}.
In Section~\ref{sec:greedy} we prove that the greedy algorithm, when combined with partial
enumeration, is a constant-factor approximation algorithm with an approximation ratio
between $(1 - \sqrt{3}/e)\approx 0.363$ and $\phi$; see \Cref{thm:greedy_approximation} and
\Cref{thm:greedy_upper_bound}.
In Section~\ref{section:monotone_greedy}, we show that a combination of the results from the previous section allows to derive a strategyproof mechanism with constant approximation ratio.
In Section~\ref{sec:randomized} we derive a randomized constant-factor approximation algorithm for the more general problem with a constant number of $r$ convex quadratic packing constraints.
The algorithm solves a convex relaxation, scales the solution, and performs randomized rounding based on that scaled solution.
Combining this algorithm with partial enumeration yields a constant-factor approximation; see \Cref{thm:rand_rounding_ratio}.
In Section~\ref{sec:hardness} we show that packing problems with convex quadratic constraints 
of type \eqref{eq:problem} are $\mathsf{APX}$-hard; see \Cref{no_ptas3}.
Finally, in
Section~\ref{sec:computational_results},
we apply the three algorithms to several instances of the problem type described in \Cref{ex:gas} based on real-world data 
from the GasLib library \cite{SABHJKKIOSSS17}.

\subsection{Related Work}

When $\vec W$ is a non-negative diagonal matrix, the quadratic constraint in \eqref{eq:problem} becomes linear and the problem is then equivalent to the $0$-$1$-knapsack problem which 
admits a \emph{fully polynomial-time approximation scheme} (FPTAS) \cite{Ibarra1975}.
Another interesting special case is when $\vec W$ is completely-positive, i.e., it can then be written as $\vec W = \vec U \vec U^{\top}$ for some matrix $\vec U \in \Qp^{n \times k}$ 
with non-negative entries. The minimal $k$ for which $\vec W$ can be expressed in this way is called the \emph{cp-rank} of $\vec W$, see \cite{Berman2003} for an overview on 
completely positive matrices. The quadratic constraint in \eqref{eq:problem} can then be expressed as $\norm{\vec U^{\top} x}_2 \leq \sqrt{c}$. 
For the case that $\vec U \in \Qp^{n \times 2}$, this problem is known as the $2$-weighted knapsack problem for which Woeginger~\cite{Woeginger2000} showed that it does not 
admit an FPTAS, unless $\mathsf{P} = \NP$. Chau et al.~\cite{Chau2016} settled the complexity of this problem showing that it admits a \emph{polynomial-time approximation scheme} (PTAS). Elbassioni et al.~\cite{Elbassioni2017} generalized this 
result to matrices with constant cp-rank.

Exchanging constraints and objective in~\eqref{eq:problem} leads to knapsack problems with quadratic objective function and a linear constraint first studied by Gallo~\cite{gallo1980}. These problems have a natural graph-theoretic interpretation where nodes and edges have profits, the nodes have weights, and the task is to choose a subset of nodes so as to maximize the total profit of the induced subgraph. Rader and Woeginger~\cite{rader2002} give an FPTAS when the graph is edge series-parallel. Pferschy and Schauer~\cite{pferschy2016} generalize this result to graphs of bounded treewidth. They also give a PTAS for graphs not including a forbidden minor which includes planar graphs.

Mechanism design problems with a knapsack constraint are contained as a special case when $W$ is a diagonal matrix. For this special case, Mu'alem and Nisan~\cite{Mualem2008} give a mechanism that is strategyproof and yields a $1/2$-approximation. Briest et al.~\cite{Briest2011} give a general framework that allows to construct a mechanism that is an FPTAS for the objective function. Aggarwal and Hartline~\cite{Aggarwal2006} study knapsack auctions with the objective to maximize the sum of the payments to the mechanism.

\section{Preliminaries}

For ease of exposition, we assume that all matrices and vectors are integer.
Let $[n] \define \{1,\dots, n\}$ and $\vec W = (w_{ij})_{i,j \in [n]}\in \N^{n\times n}$ be a symmetric psd matrix.
Furthermore, let
$\vec p\in \N^n$ be a profit vector and let $c \in \N$ be a budget. We consider problems of the form~\eqref{eq:problem}, i.e., $\max\,\{p^\top x : x^\top W x \leq c, x \in \{0,1\}^n\}$.
Throughout the paper, we denote the characteristic vector of a subset $S \subseteq [n]$ by $\vec \ind_S \in \{0,1\}^n$, i.e., $\ind_i = 1$ if $i \in S$ and $\ind_i = 0$, otherwise.

We first state the intuitive result that after fixing $x_i = 1$ for $i \in N_1 \subseteq [n]$ and fixing $x_i = 0$ for $i \in N_0$ (with $N_0 \cap N_1 = \emptyset$), we again obtain a packing problem with a convex and quadratic packing constraint.

\begin{lemma}
\label{lem:set_01}
Let $\smash{\vec W \in \N^{n \times n}}$ be symmetric psd, $\smash{\vec p \in \N^n}$, and $\smash{c \in \N}$. Further, let $N_0$, $N_1 \in \smash{2^{[n]}}$ with $N_0 \cap N_1 =  \emptyset$ and $N_0 \cup N_1 \subsetneq [n]$ be arbitrary. Then, there exist $\tilde{n} \in \N$, $\tilde{\vec W} \in \N^{\tilde{n} \times \tilde{n}}$ symmetric psd, $\tilde{\vec p} \in \N^{\tilde{n}}$, and $\tilde{c} \in \N$ such that
\begin{align*}
  &\max\,\bigl\{ \vec p^\top \vec x \,:\, \vec x^\top \vec W \vec x \leq c,\; \vec x \in \{0,1\}^n,\; x_i = 0 \;\forall i \in N_0,\; x_i =  1 \;\forall i \in N_1\bigr\}\\
  & = \vec p^\top \vec \ind_{N_1} + \max\, \bigl\{ \tilde{\vec p}^{\top}\tilde{\vec x} \,:\, \tilde{\vec x}^\top \tilde{\vec W} \tilde{\vec x} \leq \tilde{c},\; \tilde{\vec x} \in \{0,1\}^{\tilde{n}}\bigr\}.
\end{align*}
\end{lemma}
\begin{proof}
Let $n_0 = \card{N_0}$, $n_1 = \card{N_1}$, and $\tilde{n} \define n - n_0 - n_1$. 
Without loss of generality we can assume that $[\tilde n] = [n] \setminus (N_0\cup N_1)$.
Consider the matrix $\tilde{\vec W} = (\tilde{w}_{ij}) \in \N^{\tilde{n} \times \tilde{n}}$ defined as
\begin{align*}
\tilde{w}_{ij} = 
\begin{cases}
w_{ij} &\text{ if $i \neq j$},\\
w_{ij} + 2\sum_{k \in N_1} w_{ik}&\text{ if $i = j$,}
\end{cases} \qquad i,\,j\in [\tilde n].
\end{align*}
Note that $\tilde{\vec W}$ is obtained from $\vec W$ by taking principal minors and adding diagonal matrices with non-negative entries so that $\tilde{\vec W}$ is positive semi-definite. Let $\tilde{c} = c- \vec \ind_{N_1}^\top \vec W \vec \ind_{N_1}$. 
With a slight abuse of notation, for a set $S \subseteq [\tilde{n}]$, let $\tilde{\vec{\ind}}_S$ 
denote its characteristic vector in $\{0,1\}^{\tilde{n}}$ and $\vec{\ind}_S$ its characteristic vector in $\{0,1\}^n$. 
We then obtain for all $S \subseteq [\tilde{n}]$ the equality
\begin{align*}
\tilde{\vec \ind}_S^{\top} \tilde{\vec W} \tilde{\vec \ind}_S  &= \sum_{i \in S} \Big(w_{ii} + 2\sum_{k \in N_1} w_{ik}\Big) + \sum_{i,j \in S : i \neq j} 2w_{ij}\\ 
&= \vec \ind_{S \cup N_1}^\top \vec W \vec \ind_{S \cup N_1} - \vec \ind_{N_1}^\top \vec W \vec \ind_{N_1}	.
\end{align*}
Thus, we have $\tilde{\vec \ind}^\top \tilde{\vec W} \tilde{\vec \ind}_{S} \leq \tilde{c}$ 
if and only if $\vec \ind_{S \cup N_1}^\top \vec W \vec \ind_{S \cup N_1} \leq c$. 
Defining $\tilde{p} \in \N^{\tilde{n}}$ with $\tilde{p}_i = p_i$ for all $i \in [\tilde{n}]$ then establishes the claimed result.
\qed\end{proof}

By Lemma~\ref{lem:set_01}, the following assumptions are without loss of generality.

\begin{lemma}\label{lem:wlog}
It is without loss of generality to assume that $0 < w_{ii} \leq c$ and $p_i > 0$ for all $i \in [n]$.
\end{lemma}

\begin{proof}
If $w_{ii} > c$ for some $i \in [n]$, then $x_i = 0$ in every feasible solution $\vec x$. If $w_{ii} = 0$, then the positive semi-definiteness of $\vec W$ implies $w_{ij} = w_{ji} = 0$ for every $j \in [n]$. 
Hence, the value of $x_i$ does not influence the value of $\vec x^{\top}\vec W \vec x$ and it is without loss of generality to assume that $x_i = 1$.
Furthermore, if $p_i = 0$ then the value of $x_i$ does not influence the value of $p^{\top}x$ and it is without loss of generality to assume that $x_i = 0$. In all cases, Lemma~\ref{lem:set_01} yields the claimed result.
\qed\end{proof}

\section{A Golden Ratio Approximation Algorithm}
\label{cf_approx}

In this section, we derive a $\phi$-approximation algorithm for packing problems with convex quadratic constraints of type \eqref{eq:problem} where $\phi = (\sqrt{5}-1)/2 \approx 0.618$ is the inverse golden ratio. To this end, we first solve a convex relaxation of the problem. We then use the resulting solution to compute a feasible solution to another non-convex relaxation of the problem.
The second relaxation has the property that any solution can be transformed so that it has at most one fractional value, and the transformation does not decrease the objective value. Together with a partial enumeration scheme in the spirit of Sahni~\cite{Sahni1975}, this yields a $\phi$-approximation.

Denote by $\vec d \in \N^n$ the diagonal of $\vec W\in \N^{n \times n}$ and let $\smash{\vec D\define \diag(\vec d)\in \N^{n\times n}}$ be the corresponding diagonal matrix.
For a vector $\vec x\in \{0,1\}^n$ we have $x_i^2 = x_i$ for all $i \in [n]$ and, thus, we obtain
\begin{align*}
\vec x^{\top} \vec W \vec x \ge \vec x^{\top} \vec D \vec x = \vec d^{\top}\vec x \quad \text{ for all $x \in \{0,1\}^n$}.
\end{align*}
We arrive at the following relaxation of \eqref{eq:problem}:
\begin{align}
\begin{split}
\text{maximize} \qquad & \floor{\vec p^{\top} \vec x} \\
\text{subject to} \qquad & 
\vec x^{\top} \vec W \vec x \le c, \\
&\vec d^\top \vec x \leq c,\\
&\vec x \in [0,1]^n.
\end{split}\tag{$R_1$}\label{eq:convex_relax}
\end{align}

The following lemma shows that we can compute an exact optimal solution to  \eqref{eq:convex_relax} in polynomial time.

\begin{lemma}
\label{lem:convex_relax}
The relaxation \eqref{eq:convex_relax} can be solved exactly in polynomial time.
\end{lemma}

\begin{proof}
For every $\vec x \in [0,1]^n$, we have $\floor{\vec p^\top \vec x} \in  P \define \{0,\dots,\sum_{i\in [n]} p_i\}$. For fixed $q \in P$, consider the mathematical program
\begin{align}
\begin{split}
\text{minimize} \qquad & \vec x^\top \vec W \vec x \\
\text{subject to} \qquad & 
\vec p^{\top} \vec x \ge q, \\
&\vec d^\top \vec x \leq c,\\
&\vec x \in [0,1]^n
\end{split}\tag{$D_q$}\label{eq:convex_relax_dual}
\end{align}
with optimal value~$c(q)$. Since~\eqref{eq:convex_relax_dual} is quadratic and convex with linear constraints, it can be solved exactly in polynomial time, see Kozlov et al.~\cite{kozlov1980polynomial}. If $c(q) > c$, we conclude that the maximal value of \eqref{eq:convex_relax} is strictly smaller than~$q$. If $c(q) \leq c$, the corresponding solution $\vec x$ solves \eqref{eq:convex_relax} with an objective value of $q$. With binary search over $P$, we can compute the maximal value $q^* \in P$ such that \eqref{eq:convex_relax_dual} has a solution of at most $c$. The thus computed value $q^*$ is the maximal objective of \eqref{eq:convex_relax} and the corresponding optimal solution $\vec x$ of \eqref{eq:convex_relax_dual} is an optimal solution of \eqref{eq:convex_relax}.
\qed\end{proof}

We proceed to propose a second relaxation of \eqref{eq:problem}.
To this end, note that for every $\vec x\in \{0,1\}^n$ we have
\[ \vec x^{\top}\vec W \vec x = \vec x^{\top}(\vec W-\vec D) \vec x + \vec x^{\top}\vec D \vec x 
= \vec x^{\top}(\vec W-\vec D) \vec x + \vec d^{\top} \vec x.  \]
Relaxing the integrality condition yields the following relaxation of \eqref{eq:problem}:
\begin{align}
\begin{split}
\text{maximize} \qquad & \vec p^{\top} \vec x \\
\text{subject to} \qquad & 
\vec x^{\top}(\vec W-\vec D)\vec x + \vec d^{\top}\vec x \le c, \\
&\vec x \in [0,1]^n.
\end{split}
\label{eq:nonconvex_relax}\tag{$R_2$}
\end{align}
Note that since the trace of $\vec W - \vec D$ is zero, $\vec W - \vec D$ has a negative eigenvalue unless all eigenvalues are zero.
Hence, $\vec W - \vec D$ is not positive semi-definite, unless $\vec W$ is a diagonal matrix.
Therefore, the relaxation \eqref{eq:nonconvex_relax} is in general not convex.

We proceed to show that \eqref{eq:nonconvex_relax} always has an optimal solution for which at most one variable is fractional. For $\vec x\in \R^n$, let 
$N_0(\vec x) \define \{i\in [n] : x_i = 0\}$, $N_1(\vec x) \define \{i\in [n] : x_i = 1\}$, and $N_f(\vec x) \define [n] \setminus (N_1(\vec x) \cup N_0(\vec x))$.

\begin{lemma}\label{only_one_frac}
For any feasible solution $\vec x$ of \eqref{eq:nonconvex_relax}, one can construct a feasible solution 
$\bar{\vec x}$ with $\card{N_f(\bar{\vec x})} \le 1$ and $\vec p^{\top} \bar{\vec x} \ge \vec p^\top \vec x$ in linear time.
\end{lemma}

\begin{proof}
Let $\vec x$ be a feasible solution of \eqref{eq:nonconvex_relax}. 
Assume $\card{N_f(\vec x)} \ge 2$, and consider $i$, $j\in N_f( \vec x)$ with $i\neq j$, in particular, $x_i, x_j \in (0,1)$. We proceed to construct a feasible solution $\bar{\vec x}$ with $|N_f(\bar{\vec x})| \le |N_f(\vec x)| -1$ and $\vec p^\top \bar{\vec x} \ge \vec p^\top \vec x$; 
see \Cref{fig:nonconvex_constraint} for an illustration.
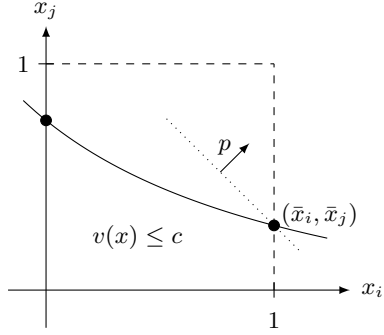
\begin{figure}[tb]
\centering
\begin{tikzpicture}
\draw[-latex] (-.5, 0) -- (4,0);
\draw[-latex] (0, -.5) -- (0,3.5);
\draw[dashed] (0,3)--(3,3);
\draw[dashed] (3,0)--(3,3);
\node at (1.2,.7) {$v(x) \le c$};
\draw (-.1,3)--(.1,3);
\draw (3,-.1)--(3,.1);
\node at (3,-.4) {1};
\node at (-.3,3) {1};
\node at (4.3,0) {$x_i$};
\node at (0,3.7) {$x_j$};
\filldraw (0,9/4) circle (2pt);
\filldraw (3,6/7) circle (2pt);
\node at (3.6, 1) {$(\bar x_i, \bar x_j)$};
\draw[rotate around={-45:(3,6/7)}, dotted] (1,6/7) -- (3.5,6/7); 
\draw[rotate around={-45:(3,6/7)}, -latex] (2,6/7) -- (2,1.4); 
\node at (2.35,1.9) {$p$};
\draw[domain=-.3:3.7, smooth, variable=\x] plot ({\x}, {(9-1*\x)/(4+1*\x)});
\end{tikzpicture}
\caption{Any feasible solution $x$ of \eqref{eq:nonconvex_relax} with $|N_f(x)|\ge 2$ can be transformed into a feasible solution $\bar x$ with $|N_f(\bar x)|\le |N_f(x)|-1$ without decreasing the objective value.}
\label{fig:nonconvex_constraint}
\end{figure}

Denote $v(x) \define \vec x^{\top}(\vec W-\vec D)\vec x + \vec d^{\top}\vec x$, and 
for $k \in \{i,j\}$ let 
\begin{align*}
\nu_k(\vec x) \define \pd{}{x_k} v(x) = \sum_{l \in [n] \setminus \{k\}} 2\,w_{kl}\,x_l + w_{kk},\qquad
r_k(\vec x) \define \frac{p_k}{\nu_k(\vec x)}. 
\end{align*}
By \Cref{lem:wlog} it is without loss of generality to assume that $w_{kk}>0$ and thus $\nu_k(\vec x) > 0$.
Note that $\nu_k(x)$ does not depend on $x_k$ and therefore, for all $\vec x\in \R^n$ and $t \in \R$, we have that 
\begin{align}
v(x + t \vec\ind_k) = v(x)  + t\, \nu_k(x),
\label{eq:translation}
\end{align}
where $\vec\ind_k\in \{0,1\}^n$ denotes the $k$-th unit vector.

Without loss of generality, assume that $r_i(\vec x) \ge r_j(\vec x)$ and define
\begin{align*}
\bar\eps &\define \frac{\nu_i(\vec x)}{\nu_j(\vec x)}(1 - x_i), &
\eps &\define \min(x_j, \bar\eps), &
\delta &\define \frac{\nu_j(\vec x)}{\nu_i(\vec x)} \eps.
\end{align*}

Consider the vector $\bar{\vec x} = \vec x - \eps \ind_j + \delta \ind_i$. By the definition of $\eps$, we have
$\bar x_j = x_j - \eps \ge  0$.
We further obtain
\begin{align*}
\bar x_i &= x_i + \delta = x_i + \frac{\nu_j(\vec x)}{\nu_i(\vec x)}\eps
\le x_i +  \frac{\nu_j(\vec x)}{\nu_i(\vec x)} \bar\eps = 1.
\end{align*}
Note that $\bar{x}_j = 0$ if $\eps = x_j$ and $\bar{x}_i = 1$ if $\eps = \bar{\eps}$ so that at least one of the inequalities $\bar{x}_j \geq 0$ and $\bar{x}_i \leq 1$ is tight.
We conclude that $\bar{\vec x} \in [0,1]^n$ and $|N_f(\bar{\vec x})| \le |N_f(\vec x)| -1$. 
Furthermore, applying \Cref{eq:translation}, we get 
\begin{align*}
v(\bar x) &= v(x - \eps \vec\ind_j + \delta \vec\ind_i) \\
&= v(x - \eps \vec\ind_j) + \delta \nu_i(x - \eps \vec\ind_j) \\
& = v(x) - \eps \nu_j(x) + \delta \nu_i(x - \eps \vec\ind_j) \\
& \le  v(x) -\eps\nu_j(x) + \delta \nu_i(x) \\
&= v(x).
\end{align*}
Thus, $\bar{\vec x}$ is a feasible solution of \eqref{eq:nonconvex_relax}. Moreover, we have
\begin{align*}
\vec p^\top \bar{\vec x} &= \vec p^{\top} \vec x - \eps p_j + \delta p_i \\
&=   \vec p^{\top}\vec x - \eps p_j + \eps \nu_j(\vec x)\frac{p_i}{\nu_i(x)}\\
&=  \vec p^{\top} \vec x - \eps p_j + \eps \nu_j(\vec x)r_i(\vec x) \\
&\ge  \vec p^{\top} \vec x - \eps p_j + \eps \nu_j(\vec x) r_j(\vec x) \\
&= \vec p^\top \vec x - \eps p_j + \eps p_j\\
&= \vec p^\top \vec x.
\end{align*}
Applying this construction iteratively (at most) $|N_f(\vec x)|-1 \le n - 1$ times yields the required result.
\qed\end{proof}

\begin{remark}\label{rem:swapping_enhancement}
The algorithm in the proof of \Cref{only_one_frac} can be improved by setting 
\begin{align*}
\bar \eps \define \frac{(1-x_i)\,\nu_i(\vec x)}{\nu_j(\vec x)+2\,w_{ij}\,(1-x_i)}, \qquad
\eps \define \min(x_j, \bar\eps), \qquad
\delta \define \frac{\eps\, \nu_j(\vec x)}{\nu_i(\vec x)-2\,w_{ij}\,x_j\,\eps}.
\end{align*}
In this way, we obtain $v(\bar x) = v(x)$ and increase the objective value at least as much as in the proof of \Cref{only_one_frac} 
while still ensuring that $\bar{\vec x}$ is feasible for \eqref{eq:nonconvex_relax} and $|N_f(\bar{\vec x})| \le |N_f(\vec x)| - 1$.
\end{remark}

We proceed to devise a $\phi$-approximation algorithm.
The algorithm iterates over all sets $H \subseteq [n]$ with $|H| \leq 3$. For each set $H$ it computes an optimal solution $\vec y^H$ to the convex relaxation \eqref{eq:convex_relax} with the additional constraints
\begin{align*}
x_i &= 1 \quad \text{ for all $i \in H$, and }\\
x_i &= 0 \quad \text{ for all $i \in \{j \in [n]\setminus H : p_j > \min_{h \in H} p_h\}$.}
\end{align*}
 Then, we scale down $\vec y^H$ by a factor of $\phi$ and show that $\phi \vec y^H$ is a 
 feasible solution to the non-convex relaxation \eqref{eq:nonconvex_relax}. 
 By Lemma~\ref{only_one_frac}, we can transform this solution into another solution $\vec z^H$ with at most one fractional variable.  
 The integral part of $\vec z^H$ is our candidate solution for the starting set $H$. 
 In the end, we return the best thus computed candidate over all possible sets $H$; see Algorithm \ref{golden_ratio_algo}.

\let\oldnl\nl
\newcommand{\nonl}{\renewcommand{\nl}{\let\nl\oldnl}}
\begin{algorithm}[tb]
\ForEach{$H \subseteq [n]$ with $|H| \le 3$}{
	$\vec y^H \assign$ sol.\ of \eqref{eq:convex_relax} with  $x_i = 1 \,\forall i \in H$, $x_i = 0 \, \forall i \in \{j \in [n]\setminus H : p_j > \smash{\min\limits_{h \in H} p_h}\}$\;
	$\vec z^H \assign$ transf.\ of $\phi \vec y^H$ containing at most one fractional variable\;
	$\bar{\vec z}^H \assign \floor{\vec z^H}$\;}
$H^* \assign \argmax\, \{ \vec p^\top \bar{\vec z}^H : H \subseteq [n] \text{ with } |H| \le 3\}$\;
\textbf{return} $\bar z^{H^*}$\;
\caption{Golden ratio algorithm\label{golden_ratio_algo}}
\end{algorithm}

\begin{theorem}\label{thm:phi_approximation}
Algorithm \ref{golden_ratio_algo} computes a $\phi$-approximation for \eqref{eq:problem}.	
\end{theorem}

\begin{proof}
Fix an optimal solution $\vec x^*$ of \eqref{eq:problem}
and define $S^* \define \{i \in [n] : x_i^* = 1\}$.
Since the algorithm iterates over all solutions of size at least three, 
it is without loss of generality for our following arguments to assume 
that $|S^*| \ge 4$.
Let $H^* \subset S^*$ with $|H^*| = 3$ be chosen such that $p_i \leq \min_{h \in H^*} p_h$ 
for all $i \in S^*$ and consider the run of the algorithm when starting with $H^*$.
Let $\bar{H} \define \{i \in [n]\setminus H^* : p_i > \min_{h \in H^*} p_h\}$ and $k \define \card{\bar{H}}$. 
It is without loss of generality to assume that $[n] \setminus (H^* \cup \bar H) = [n-k-3]$.
Consider the packing problem where as additional constraints we have $x_i = 1$ for all $i \in H^*$ and $x_i = 0$ for all $i \in \bar{H}$. 
By Lemma~\ref{lem:set_01}, this packing problem can be written as
\begin{align*}
\begin{split}
\text{maximize} \qquad & \tilde{\vec p}^{\top} \vec x \\
\text{subject to} \qquad &  
\vec x^{\top} \tilde{\vec W} \vec x \le \tilde{c}, \\
&\vec x \in \{0,1\}^{n-k-3},
\end{split}\tag{$\tilde{P}$}\label{Ptilde}
\end{align*}
where $\tilde{\vec W}$ is a symmetric and positive semi-definite matrix. We then have $\vec p^\top \vec x^* = \sum_{h \in H^*} p_i + \tilde{\vec p}^{\top} \tilde{\vec x}^*$ for an optimal solution $\tilde{\vec x}^*$ of \eqref{Ptilde}.

Let $\vec y$ be an optimal solution to the convex relaxation \eqref{eq:convex_relax} of \eqref{Ptilde}. Since \eqref{eq:convex_relax} is a relaxation of \eqref{Ptilde}, we have $\vec p^\top \vec y \geq \vec p^\top \tilde{\vec x}^*$. 
We proceed to show that $\phi \vec y$ is feasible for the non-convex relaxation \eqref{eq:nonconvex_relax} of \eqref{Ptilde}. To this end, we calculate
\begin{align*}
(\phi\vec y)^\top (\tilde{\vec W} - \tilde{\vec D}) (\phi\vec y) + \phi\, \tilde{\vec d}^\top \vec y = \phi^2  \vec y^\top (\tilde{\vec W} - \tilde{\vec D}) \vec y + \phi\, \tilde{\vec d}^\top \vec y
\leq \phi^2 \tilde{c} + \phi \tilde{c}
= \tilde{c},
\end{align*}
where for the inequality we used that $\vec y$ is feasible for the convex relaxation and, 
thus, $\vec y^\top \tilde{\vec W} \vec y \leq \tilde{c}$ and $\tilde{\vec d}^\top \vec y \leq \tilde{c}$. By Lemma~\ref{only_one_frac}, we can transform $\phi\vec y$ into a solution $\vec z$ such that $\tilde{\vec p}^\top \vec z \geq \phi\tilde{\vec p}^{\top} \vec y$ and $\vec z$ has at most one fractional variable $z_{\ell}$ with $\ell \in [n-k-3]$.

Let $S = H^* \cup \{i \in [n-k-3] : z_i = 1\}$ and consider the solution $\vec \ind_S$. We have that
\begin{align*}
\vec \ind_S^\top \vec W \vec \ind_S = \vec \ind_{H^*}^\top \vec W \vec \ind_{H^*} + \vec z^\top \tilde{\vec W} \vec z
\leq  \vec \ind_{H^*}^\top \vec W \vec \ind_{H^*}  + \tilde{c}
= c,
\end{align*}
so that $\vec \ind_S$ is feasible for \eqref{eq:problem}. Moreover, we obtain
{\allowdisplaybreaks
\begin{align*}
\vec p^\top \vec \ind_S &= \sum_{h \in H^*} p_h + \tilde{\vec p}^\top	\vec z - p_{\ell} z_{\ell} \\
&\geq \sum_{h \in H^*} p_h + \phi\tilde{\vec p}^\top \vec y - p_{\ell}\\
&\geq \sum_{h \in H^*} p_h + \phi\tilde{\vec p}^\top \tilde{\vec x}^* - p_{\ell} \\
&= \phi \vec p^\top \vec x^* + (1-\phi) \sum_{h \in H^*} p_h - p_{\ell}\\
&\geq \phi \vec p^\top \vec x^* + (3(1-\phi)-1)\min_{h \in H^*} p_h\\
&\geq \phi \vec p^\top \vec x^*,
\end{align*}
}
establishing the claimed result.
\qed\end{proof}

As a result of \Cref{thm:phi_approximation}, we can derive an upper bound on the optimal value of \eqref{eq:convex_relax}.
This will turn out to be useful when constructing a monotone greedy algorithm in the next section.

\begin{corollary}\label{cor:upper_bound_for_relax}
Let $x^*$ and $y^*$ be optimal solutions to \eqref{eq:problem} and \eqref{eq:convex_relax}, respectively.
Then $p^{\top}y^* \le \frac{2}{\phi}p^{\top} x^*$.
\label{relax_lower_bound}
\end{corollary}

\begin{proof}
Since $y^*$ is feasible for \eqref{eq:convex_relax}, we have
\[(\phi y^*)^{\top}(W-D)(\phi y^*) + d^{\top}(\phi y^*) \le \phi^2 (y^*)^TWy^* + \phi d^{\top}y^* \le (\phi^2 + \phi)c = c.\] 
Therefore, $\phi y^*$  is 
feasible for \eqref{eq:nonconvex_relax}. By \Cref{only_one_frac}, we can transform $\phi y^*$ into a vector $z$ with $p^{\top} z \ge p^{\top}(\phi y^*) = \phi p^{\top} y^*$ and
$\card{N_f(z)} \le 1$. The integral part~$\floor{z}$ of~$z$ is feasible for \eqref{eq:problem}, and thus,
$p^{\top} z \le p^{\top} \floor{z} + \max_{i\in [n]} p_i \le 2p^{\top} x^*$.
We conclude that $p^{\top}y^* \le \frac{1}{\phi}p^{\top}z \le \frac{2}{\phi}p^{\top}x^*$.
\qed\end{proof}

\section{The Greedy Algorithm}
\label{sec:greedy}

In this section we analyze the greedy algorithm and show that, when combined with a partial enumeration scheme in the spirit of Sahni~\cite{Sahni1975},
it is at least a $(1 - \sqrt{3}/e)$-approximation for packing problems with quadratic constraints of type \eqref{eq:problem}. 
We further show that its approximation ratio can be bounded from above by the golden ratio $\phi$.
Even though this approximation ratio is thus not better than the one guaranteed by the golden ratio algorithm (Theorem~\ref{thm:phi_approximation}),
it is worth analyzing it for several reasons. Firstly, it is simple to understand as well as to implement 
and turns out to have a much better running time in practice than the golden ratio algorithm; see the computational results in
\Cref{sec:computational_results}.
And, secondly, the greedy algorithm serves as a main building block to devise a strategyproof mechanism with constant welfare guarantee; see Section~\ref{section:monotone_greedy}.

For a set $S\subseteq [n]$, we write $w(S) \define \vec \ind_S^\top \vec W \vec \ind_S$.
The core idea of the greedy algorithm is as follows.
Assume that we have an initial solution $S \subset [n]$.
Amongst all remaining items in $[n] \setminus S$, we pick an item $i$ that maximizes the ratio between profit gain and weight gain, 
i.e., 
\[
  i \in \argmax_{j\in [n]\setminus S}\; \frac{p_j}{w(S\cup \{j\}) - w(S)}.
\] 
If adding $i$ to the solution set would make it infeasible, i.e., $w(S \cup \{i\}) > c$, then we delete $i$ from $[n]$.
Otherwise, we add $i$ to $S$.
We repeat this process until $[n] \setminus S$ is empty.  

It is known from the knapsack problem that, when starting the greedy algorithm as described above with the empty set as initial set,
then the produced solution can be arbitrarily bad compared to an optimal solution.
However, the greedy algorithm can be turned into a constant-factor approximation by using partial enumeration:
For all feasible subsets $U\subseteq [n]$ with $|U| \le 2$, we run the greedy algorithm starting with $U$ as initial set.
In the end we return the best solution set found in this process; 
see Algorithm \ref{alg:greedy}.

\begin{algorithm}[tb]
\ForEach{$U \subseteq [n]$ with $|U| \le 2$}{
	$S \assign U$, $I \assign [n]$ \;
	\While{$I \setminus S \neq \emptyset$}{
		$i \assign \argmax_{j \in I \setminus S} \frac{p_j}{w(S\cup \{j\} )-w(S)}$\;
		\eIf{$w(S \cup \{i\}) \leq c$}{
			$S \assign S \cup \{i\}$\;
		}{
			$I \assign I \setminus \{i\}$\;
		}
	}
	$S_U \assign S$\;
}
$U^*\assign \argmax\, \{\vec p^{\top} \vec \ind_{S_U} : U \subseteq [n] : \card{U} \le 2\}$\;
\textbf{return} $\ind_{S_{U^*}}$
\caption{Greedy algorithm with partial enumeration\label{alg:greedy}}	
\end{algorithm}

The analysis of the algorithm follows a similar approach as the analysis of Sviridenko~\cite{Sviridenko2004} for the greedy algorithm for maximizing a submodular function under a linear knapsack constraint. The non-linearity of the constraint in our case makes the analysis more complicated, though. 
In order to prove the approximation ratio of the greedy algorithm we need the following two technical lemmas.

\begin{lemma}\label{lem:sequence_bounded}
Let $m\in \N$ and consider the sequence $(\theta_t)_{t\in \N}$ defined by the recursive formula
\[ 1 - \big(m + 2\sqrt{tm} \big)\theta_{t+1} =\sum_{i=1}^t\theta_i, \quad \theta_1 = \frac{1}{m}. \]
Then $\sum_{t=1}^m \theta_t \ge 1 - \frac{\sqrt{3}}{e}$. 
\end{lemma}

\begin{proof}
Consider the initial value problem 
\begin{align*}
\psi'(x) = \frac{1 - \psi(x)}{1 + 2\sqrt{x}}, \quad x\in [0,1],\quad \psi(0) = 0.
\end{align*}
Since the function $f\colon [0,1] \times \R \to \R$, $(x,s) \mapsto \frac{1-s}{1 + 2\sqrt{x}}$ is Lipschitz-continuous in~$s$, 
by the Picard-Lindel\"of Theorem, this problem has a unique solution, which is given by 
\[
\psi(x) = 1 - e^{-\sqrt{x}}\sqrt{1 + 2\sqrt{x}}, \quad x\in [0,1].
\] 
Since its first derivative 
\[
\psi'(x) = \frac{e^{-\sqrt{x}}}{\sqrt{1 + 2\sqrt{x}}}, \quad x\in [0,1],
\]
is monotonically decreasing, it follows that $\psi$ is concave.

Define $z_t \define \sum_{i=1}^t\theta_i$, $t\in \{0, \dots, m\}$.
We claim that for every $t \in \{0, \dots, m\}$ we have
\begin{align}
\label{eq:psi}
z_t \ge \psi\big(\tfrac{t}{m}\big).  
\end{align}
Note that \eqref{eq:psi} implies the result using
\[
\sum_{t=1}^m \theta_t = z_{m}\ge \psi(1) = 1 - \frac{\sqrt{3}}{e}.
\] 

To finish the proof, we prove \eqref{eq:psi} by induction. We have $z_0 = 0 = \psi(0)$. 
Now assume that  \eqref{eq:psi} holds for some arbitrary but fixed $t\in \{0,\dots,m-1\}$.
By the recursive definition of $(\theta_t)_{t\in \N}$ and the concavity of $\psi$, it then follows that 
\begin{align*}
z_{t+1} &= z_t + \theta_{t+1} \\
&= z_t + \frac{1 - z_t}{m + 2\sqrt{tm}} \\
&= z_t \left( 1 - \frac{1}{m + 2\sqrt{tm}} \right) + \frac{1}{m + 2\sqrt{tm}} \\
&\ge \psi\big(\tfrac{t}{m}\big)\left( 1 - \frac{1}{m + 2\sqrt{tm}} \right) + \frac{1}{m + 2\sqrt{tm}} \\
&= \psi\big(\tfrac{t}{m}\big) + \frac{1 - \psi(\frac{t}{m})}{m + 2\sqrt{tm}} \\
& = \psi\big(\tfrac{t}{m}\big) + \tfrac{1}{m} \psi'\big(\tfrac{t}{m}\big) \\
&\ge \psi\big( \tfrac{t+1}{m} \big),
\end{align*}
which completes the proof. 
\qed\end{proof}

\begin{lemma}
\label{lem:ratio_bound}
Let $w_0,\dots, w_m\in \N$ with $0 = w_0 < w_1 < \dots < w_m$,  
and let $\theta_i \ge 0$, $i\in [m]$. Then,
\begin{align*}
\sum_{i=1}^m \theta_i (w_i \!-\! w_{i-1}) 
\ge \biggl(1 \!-\! \frac{\sqrt{3}}{e} \biggr) 
\min_{t=0,\dots,m\!-\!1} \sum_{i=1}^t \theta_i(w_i \!-\!w_{i-1}) + \theta_{t+1}(w_m \!+\! 2\sqrt{w_tw_m}).
\end{align*}
\end{lemma}

\begin{proof}
We first show the statement for sequences $0 = w_0 < w_1 < \dots < w_m$ with the additional property that $w_i - w_{i-1} = 1$ for all $i \in [m]$.
For this case, it is to show that 
\begin{align*}
\sum_{i=1}^m \theta_i \ge \bigg( 1 - \frac{\sqrt{3}}{e} \bigg) 
\min_{t=0,\dots,m-1} \sum_{i=1}^t \theta_i + \theta_{t+1}\big(m + 2\sqrt{tm}\big).
\end{align*}
It suffices to show that the optimal value of the following optimization problem is at least
$1 - \sqrt{3}/e$. 
\begin{align}
\begin{split}
\text{minimize} \qquad & \sum_{i=1}^m \theta_i\\
\text{subject to} \qquad & \sum_{i=1}^t \theta_i + \theta_{t+1}(m + 2\sqrt{tm}) \ge 1 \quad \text{for all }t=0,\dots, m-1, \\
& \theta_i \ge 0 , \qquad \text{for all }i =1,\dots, m.
\label{eq:primal}
\end{split}
\end{align}
We claim that every optimal solution to \eqref{eq:primal} satisfies all inequalities with equality. For a contradiction, fix an optimal solution $\theta^*_1,\dots,\theta_m^*$ and suppose there is $s \in \{0,\dots,m-1\}$ such that
\begin{align*}
\sum_{i=1}^s	 \theta_i^* + \theta_{s+1}^*(m + 2\sqrt{sm}) > 1.	
\end{align*}
Choosing the minimal $s$ with this property, we have $\theta_{s+1}^* > 0$.
Let
\begin{align*}
\delta =\min \left\{ \theta_{s+1}^* ,   \frac{\sum\limits_{i=1}^s \theta_i^* + \theta_{s+1}^*(m + 2\sqrt{sm}) - 1}{m+2\sqrt{sm}} \right\},
\end{align*}
and consider the solution $\theta_1',\dots,\theta_m'$ defined as
\begin{align*}
\theta_i' =
\begin{cases}
\theta_i^* & \text{ if } i < s+1,\\  
\theta_i^* - \delta & \text{ if } i = s+1,\\
\theta_i^* + \frac{\delta}{m+2\sqrt{(i-1)m}} & \text{ if } i > s+1.\\
\end{cases}
\end{align*}
We first check that the solution $\theta_1',\dots,\theta_m'$ is feasible. For the inequalities for $t=0,\dots,s-1$, there is nothing to show since the involved variables are not altered. For $t=s$, the inequality is satisfied by the choice of $\delta$. For $t > s$, we obtain
\begin{align*}
\sum_{i=1}^t \theta_i' + \theta'_{t+1}(m + 2\sqrt{tm}) \geq \sum_{i=1}^t \theta_i^* - \delta + \biggl(\theta^*_{t+1} + \frac{\delta}{m + 2 \sqrt{tm}}\biggr)(m + 2\sqrt{tm})\geq 1,
\end{align*}
where for the second inequality we used that $\theta_1^*,\dots,\theta_m^*$ is feasible. Finally, we note that
\begin{align*}
\sum_{i=1}^m \theta_i^* - \sum_{i=1}^m \theta_i' = \delta - \sum_{i=s+2}^m \frac{\delta}{m+2\sqrt{(i-1)m}}
\geq \delta\biggl(1 - \frac{m-1}{m}\biggr)
> 0,
\end{align*}
contradicting the optimality of $\theta^*_1,\dots,\theta^*_m$. We conclude that every optimal solution of \eqref{eq:primal} satisfies all inequalities with equality. The result then follows from \Cref{lem:sequence_bounded}.

It is left to show that the statement holds for arbitrary finite sequences 
$0 = w_0 < w_1 < \dots < w_m$. Fix such a sequence, let $m' \define w_m$, 
and let $\theta_1',\dots, \theta_{m'}$ be such that there are first $w_1 - w_0$ 
copies of $\theta_1$, then $w_2 - w_1$ copies of $\theta_2$, and so on.
We thus obtain
\begin{align*}
\sum_{i=1}^m & \theta_i(w_i - w_{i-1}) = \sum_{i=1}^{m'} \theta_i' \\
&\ge \biggl(1 - \frac{\sqrt{3}}{e} \biggr)\min_{t=0,\dots,m'-1}\sum_{i=1}^{t} \theta_i' + 
\theta'_{t+1}(m' + 2\sqrt{t m'}) \\
&= \biggl(1 - \frac{\sqrt{3}}{e} \biggr)\min_{t=w_0,\dots,w_{m-1}}\sum_{i=1}^{t} \theta_i' + 
\theta'_{t+1}(m' + 2\sqrt{t m'}) \\
&= \biggl(1 - \frac{\sqrt{3}}{e} \biggr)\min_{t=0,\dots,m-1}\sum_{i=1}^{t} (w_i - w_{i-1})\theta_i + 
\theta_{t+1}(w_m + 2\sqrt{t w_m}),
\end{align*}
yielding the result.
\qed\end{proof}

We can now prove the approximation ratio of the greedy algorithm.

\begin{theorem}\label{thm:greedy_approximation}
The Greedy algorithm with partial enumeration 
(Algorithm~\ref{alg:greedy})
is a an approximation algorithm with approximation ratio $(1 - \frac{\sqrt{3}}{e})$ for~\eqref{eq:problem}.	
\end{theorem}

\begin{proof}
Let $x^*$ be an optimal solution of \eqref{eq:problem} and set 
$S^* \define \{i\in [n]: x^*_i = 1\}$. Number the items of $S^* =
\{i_1^*,i_2^*,\dots, i_k^*\}$ such that $p_{i_1^*} \geq p_{i_2^*} \geq \dots \geq p_{i_k^*}$.
Since the algorithm enumerates all solutions with at most two items, it is without loss of generality to assume that $|S^*|\ge 3$. 
Consider the run of the greedy algorithm with $U = \{i_1^*,i_2^*\}$. Without loss of generality, we assume that $i_1^* = n-1$ and $i_2^* = n$. 
Set $S^0\define U$, and for $t = 1,2,\dots$, denote by $S^t$ and $i_t$ the values of $S$ and $i$ after the $t$-th pass of the \texttt{while} loop. 
Furthermore, define 
\[ \theta_{t}\define \frac{p_{i_{t}}}{w(S^{t-1}\cup \{i_{t}\}) - w(S^{t-1})}. \] 

By Lemma~\ref{lem:set_01}, we can treat the problem after fixing $x_{i_1^*} = x_{i_2^*}= 1$ as a new problem of the same form with matrix $\smash{\tilde{\vec W} \in \N^{(n-2) \times (n-2)}}$, profit vector $\smash{\tilde{\vec p}}$, and budget $\smash{\tilde{c}}$. 
In the following, for a set $S \subseteq [n] \setminus U$ we write $\tilde{w}(S) \define \vec \ind_{S}^{\top} \tilde{\vec W} \vec \ind_S$.
Note that $\tilde w$ is supermodular, i.e., for any two sets $S$, $S' \subseteq [n]\setminus U$ we have
\[
\sum_{i\in S'\setminus S} \tilde w(S\cup \{i\}) - \tilde w(S) \le \tilde w(S\cup S') - \tilde w(S).
\] 
By \Cref{lem:wlog}, we can assume without loss of generality that $\tilde w(S\cup \{i\}) - \tilde w(S) > 0$.

Let $t^*$ be the first step of the greedy algorithm for which $i_{t^*} \in S^*$ but
the algorithm does not add $i_{t^*}$ to its solution set. 
It is without loss of generality to assume that in all previous iterations $t \in \{1,\dots,t^*-1\}$ we had $S^{t} = S^{t-1} \cup \{i_t\}$ as otherwise item $i_t$ would be neither contained in the optimal solution nor the solution computed by the greedy algorithm; thus, removing it from the instance would not change the analysis. Since $i_{t^*}$ is not included in the solution, we have $\tilde{w}(S^{t^*-1}\cup \{i_{t^*}\} \setminus U) > \tilde{c}$. 
In the following, we write $S^{t^*} \define S^{t^*-1} \cup \{i_{t^*}\}$,  $\tilde{S}^* \define S^* \setminus U$,
and for $t \in \{0,\dots,t^*\}$, we write $\tilde{S}^t \define S^t \setminus U$.

For all $t \in \{0,\dots,t^*-1\}$, we obtain
\begin{align*}
\sum_{i \in \tilde{S}^*} p_i &\leq \sum_{i \in \tilde{S}^t} p_i + \sum_{i \in \tilde{S}^*\setminus \tilde{S}^t} p_i\\
&= \sum_{i \in \tilde{S}^t} p_i + \sum_{i \in \tilde{S}^*\setminus \tilde{S}^t} \frac{p_i}{\tilde{w}(\tilde{S}^{t} \cup \{i\}) - \tilde{w}(\tilde{S}^{t})}\big(\tilde{w}(\tilde{S}^{t} \cup \{i\}) - \tilde{w}(\tilde{S}^{t})\big)\\
&\leq \sum_{i \in \tilde{S}^t} p_i + \theta_{t+1} \sum_{i \in \tilde{S}^*\setminus \tilde{S}^t} \big(\tilde{w}(\tilde{S}^{t} \cup \{i\}) - \tilde{w}(\tilde{S}^{t})\big)\\
&\leq \sum_{i \in \tilde{S}^t} p_i + \theta_{t+1} \big(\tilde{w}(\tilde{S}^{t} \cup \tilde{S}^*) - \tilde{w}(\tilde{S}^{t})\big),
\end{align*}
where we used the supermodularity of $\tilde{w}$. By the Cauchy-Schwarz inequality it holds that
\begin{align*}
\tilde{w}(\tilde{S}^{t} \cup \tilde{S}^*) - \tilde{w}(\tilde{S}^{t}) &= (\ind_{\tilde{S}^t} + \ind_{\tilde{S}^* \setminus S^t})^\top \tilde{W}	(\ind_{\tilde{S}^t} + \ind_{\tilde{S}^* \setminus \tilde{S}^t}) - w(\tilde{S}^t)\\
&\leq \tilde{w}(\tilde{S}^* \setminus \tilde{S}^t) +2 \sqrt{\tilde{w}(\tilde{S}^t) \tilde{w}(\tilde{S}^*\setminus \tilde{S}^t)}\\
&\leq \tilde{c} + 2\sqrt{\tilde{w}(\tilde{S}^t)\tilde{c}}.
\end{align*}
Thus, we get
\begin{align*}
\sum_{i \in \tilde{S}^*} p_i &\leq \sum_{i \in \tilde{S}^t} p_i + \theta_{t+1}\Bigl( \tilde{c} + 2\sqrt{\tilde{w}(\tilde{S}^t)\tilde{c}} \Bigr) 
\quad \text{for all } t \in \{0,\dots,t^*-1\}. 
\end{align*}
Since $\tilde c < \tilde w(\tilde S^{t^*})$, it follows that
\begin{align*}
\sum_{i \in \tilde{S}^*} p_i 
&\leq \min_{t = 0,\dots,t^*-1}\sum_{i \in \tilde{S}^t} p_i + \theta_{t+1}\Big(\tilde w(\tilde S^{t^*}) + 2\sqrt{\tilde w(\tilde S^{t^*})\tilde w(\tilde S^t})\Big) \nonumber\\
&= \min_{t = 0,\dots,t^*}\sum_{i = 1}^t \theta_i\Big(\tilde w(\tilde S^i) - \tilde w(\tilde S^{i-1}) \Big)
+ \theta_{t+1}\Big(\tilde w(\tilde S^{t^*}) + 2\sqrt{\tilde w(\tilde S^{t^*})\tilde w(\tilde S^t})\Big).
\end{align*}
Furthermore, 
 \begin{align*}
 \sum_{i \in \tilde{S}^{t^*}} p_i = \sum_{i = 1}^{t^*}\theta_i\Big(\tilde w(\tilde S^i) - \tilde w(\tilde S^{i-1}) \Big)
 \end{align*}
 and thus, by Lemma~\ref{lem:ratio_bound}, 
 \[ \sum_{i \in \tilde{S}^{t^*}} p_i \ge \left( 1 - \frac{\sqrt{3}}{e} \right) \sum_{i \in \tilde{S}^*} p_i.  \] 
Finally, this leads to 
\begin{align*}
\sum_{i \in S^{t^*-1}} p_i &= \sum_{i \in U} p_i + \sum_{i \in \tilde{S}^{t^*-1} } p_i\\
&= \sum_{i \in U} p_i + \sum_{i \in \tilde S^{t^*}} p_i - p_{i_{t^*}} \\
&\geq \sum_{i \in U} p_i + \biggl(1- \frac{\sqrt{3}}{e}\biggr)\sum_{i \in \tilde{S}^*} p_i -p_{i_{t^*}} \\
&\geq \sum_{i \in U} p_i + \biggl(1- \frac{\sqrt{3}}{e}\biggr)\sum_{i \in \tilde{S}^*} p_i -\frac{1}{2}\sum_{i \in U} p_i\\
&\geq \biggl(1- \frac{\sqrt{3}}{e}\biggr)\sum_{i \in S^*} p_i.
\end{align*}
Since the greedy algorithm with starting solution $U$ obtains a profit of at least $\sum_{i \in S^{t^*-1}} p_i$, this implies the claimed result.
\qed\end{proof}

We proceed to show that the approximation ratio of the greedy algorithm can be bounded from above by the golden ratio.

\begin{theorem}\label{thm:greedy_upper_bound}
The approximation ratio of the greedy algorithm with partial enumeration is at most $\phi = (\sqrt{5} -1)/2$,
even if we allow partial enumeration over an arbitrary but fixed number of items. 
\label{greedy_upper_bound}
\end{theorem}

\begin{proof}
Consider the following instance. Let $m$, $\ell$, $k\in \N$ with $\ell < k$ and denote by $\ind_i$ the $i$-th unit vector in $\R^m$.
Let there be two types of items: $m$ items of type 1 with profit $p^{(1)} =  1$ and weight vector $y^{(1)}_i = k\ind_i$, $i\in [m]$,
and $m\ell$ type~2 items with profit $p^{(2)} = \frac{1+2\ell}{k^2+2k\ell}$ and weight vector $y^{(2)}_i = \ind_{\lceil \frac{i}{\ell} \rceil }$,
$i\in [m\ell]$; see \Cref{fig:greedy_solution} for an illustration.
\begin{figure}[tb]
\centering
\pgfmathsetmacro{\x}{.2}
\pgfmathsetmacro{\y}{.2}
\begin{tikzpicture}[xscale = .5, yscale = .3]

\draw[dotted] (-1,6) node[left]{$k$}  -- (0+\x,6);
\draw[dotted] (6-\x,4) -- (7,4) node[right]{$\ell$} ;

\draw[fill=white] (0+\x,0) rectangle (1-\x,6);
\draw[fill=white] (1+\x,0) rectangle (2-\x,6);

\draw (2+\x,0) rectangle (3-\x,1);
\draw (2+\x,1) rectangle (3-\x,2);
\draw (2+\x,2) rectangle (3-\x,3);
\draw (2+\x,3) rectangle (3-\x,4);

\draw[fill=white] (2+\x,4) rectangle (3-\x,10);

\draw (3+\x,0) rectangle (4-\x,1);
\draw (3+\x,1) rectangle (4-\x,2);
\draw (3+\x,2) rectangle (4-\x,3);
\draw (3+\x,3) rectangle (4-\x,4);

\draw (4+\x,0) rectangle (5-\x,1);
\draw (4+\x,1) rectangle (5-\x,2);
\draw (4+\x,2) rectangle (5-\x,3);
\draw (4+\x,3) rectangle (5-\x,4);

\draw (5+\x,0) rectangle (6-\x,1);
\draw (5+\x,1) rectangle (6-\x,2);
\draw (5+\x,2) rectangle (6-\x,3);
\draw (5+\x,3) rectangle (6-\x,4);

\node at (.5,-1) {$1$};
\node at (1.5,-1) {$2$};
\node at (3.5,-1) {$\dots$};
\node at (5.5,-1) {$m$};
\draw[thick] (-1,0) -- (7,0);

\end{tikzpicture}
\caption{A partial greedy solution $S = (S_1,S_2)$ with initial set $S^0 = (S^0_1,S^0_2)$, where $S^0_1 = \{1,2\}$ and $S^0_2 = \emptyset$.
The long bars represent type 1 items whereas the short bars represent type 2 items.}
\label{fig:greedy_solution}
\end{figure}
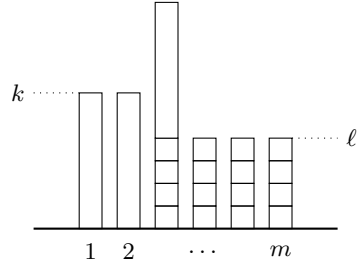

We wish to solve
\begin{align*}
\text{maximize} \qquad & p(S_1,S_2) \define |S_1|\, p^{(1)} +|S_2|\,p^{(2)} \\
\text{subject to} \qquad &
w(S_1,S_2) \define \Big\lVert\sum_{i\in S_1}y_i^{(1)} + \sum_{i\in S_2}y_i^{(2)} \Big\rVert^2_2 \le mk^2, \\
&S_1\subseteq [m], \; S_2\subseteq [m\ell].
\end{align*}
Setting $Y \define [y_1^{(1)},\dots, y_m^{(1)},y_1^{(2)},\dots,y_{m\ell}^{(2)}]$,
this optimization problem can be reformulated as in~\eqref{eq:problem} with 
weight matrix $W = Y^{\top}Y$, which is clearly nonnegative and positive semidefinite.

We first derive the solution produced by the greedy algorithm.
Partition $[m\ell] = \bigcup_{i=1}^{m}T_i$, where $T_i = \{j\in [m\ell]: \lceil \frac{j}{\ell}\rceil = i\}$.
Consider a partial greedy solution $S = (S_1,S_2)$ and assume that $i\notin S_1$ and $j \notin S_2$
for some type 1 item $i\in [m]$ and type 2 item $j\in T_i$. Let $h := |S_2\cap T_i| < \ell$. 
Then we have $w(S_1,S_2\cup\{j\}) - w(S_1,S_2)$ = $(h+1)^2 - h^2 = 1 + 2h$, and thus
\begin{align*}
\frac{p^{(2)}}{w(S_1,S_2\cup\{j\}) - w(S_1,S_2)} &= \frac{1+2\ell}{(k^2+2k\ell)(1+2h)} \\
&> \frac{1+2\ell}{(k^2+2kh)(1+2\ell)} \\
&= \frac{1}{k^2 + 2kh} \\
&= \frac{p^{(1)}}{w(S_1\cup\{i\},S_2) - w(S_1,S_2)}.
\end{align*}
Hence, the greedy algorithm will always include type 2 item $j\in T_i$ before type~1 item $i$ in its solution.

Assume that for a partial solution $S = (S_1,S_2)$ we have $i\in S_1$, $i'\notin S_1$, and $j\notin S_2$ for 
some type 1 items $i$, $i'\in [m]$ and a type 2 item $j\in T_i$. Since $|S_2 \cap T_{i'}| \le \ell$, we have
\begin{align*}
\frac{p^{(1)}}{w(S_1\cup \{i'\},S_2) - w(S_1,S_2)} &\ge \frac{1}{k^2 + 2k\ell} \\
&> \frac{1+2\ell}{(1+2k)(k^2+2k\ell)} \\
&\ge \frac{p^{(2)}}{w(S_1,S_2\cup \{j\})-w(S_1,S_2)}.
\end{align*}
Consequently, the greedy algorithm will always add type 1 item $i'$ before type 2 item $j\in T_i$ to its solution
given that type 1 item $i$ is already included.

Thus, the greedy algorithm starts with some initial solution $S^0 = (S_1^0,S_2^0)$.
Afterwards, it includes all type 2 items in $[m\ell] \setminus \bigcup_{i\in S_1^0} T_i$ (Step 1).
Finally, it adds type 1 items until the capacity bound of $mk^2$ is reached (Step 2).
Let $s \define |S_1^0|$. The weight of the partial solution after Step 1 is given by $sk^2 + (m-s)\ell^2$.
Adding any type~1 item in Step 2 increases the weight of the solution by $k^2 + 2k\ell$.
Hence, in Step 2,
\[ r:= \frac{mk^2 - sk^2 - (m-s)\ell^2}{k^2 + 2k\ell} = \frac{(m-s)(k^2 - \ell^2)}{k^2 + 2k\ell}\]
type 1 items are added until the capacity is reached. (It is without loss of generality to assume that $r\in \Z$ since otherwise after adding $\floor{r}$ type 1 items, the remaining capacity would be filled with type 2 items and the resulting approximation ratio would be even lower.)
Thus, the profit of the solution $\hat S$ produced by the greedy algorithm is given by
\begin{align*}
p(\hat S) &= (s+r)p^{(1)} + (m-s)\ell p^{(2)} \\
&= s + \frac{(m-s)(k^2 - \ell^2)}{k^2 + 2k\ell} + \frac{(m-s)\ell (1+2\ell)}{k^2+2k\ell} \\
&= s + \frac{(m-s)(k^2 + \ell^2 + \ell)}{k^2+2k\ell} \\
&= m\left[ \frac{s}{m} + \frac{ (1 - \frac{s}{m})(1 + \frac{\ell^2}{k^2} + \frac{\ell}{k^2})}{1  + 2\frac{\ell}{k}} \right] \\
&= m\left[ \frac{s}{m} + \frac{(1 - \frac{s}{m}) (1+ q^2 + \frac{q}{k})}{1 + 2q} \right],
\end{align*}
where $q:= \frac{\ell}{k}$.

On the other hand, consider the solution $S^* = (S_1^*,S_2^*)$ with $S_1^* = [m]$ and $S_2^* = \emptyset$.
It fulfills $p(S) = m$ and $w(S) = mk^2$.
Thus, we have
\begin{align*}
\rho(q):= \lim_{k,m\to \infty} \frac{p(\hat S)}{p(S^*)} = \frac{1 + q^2}{1 + 2q}.
\end{align*}
Under the constraint $q\in (0,1)$, the ratio $\rho(q)$ attains its minimum at $q = \phi$ with value $q(\phi) = \phi$, where $\phi = \frac{\sqrt{5}-1}{2}$ is the golden ratio.
\qed
\end{proof}

\section{Monotone Algorithms}
\label{section:monotone_greedy}

To illustrate the need for monotone algorithms, reconsider the situation described in Example~\ref{ex:gas} with a set of $n$ selfish agents requesting permission to send gas through a pipeline.
Each agent~$j$ has a private value $p_j$ expressing the monetary gain from being allowed to send the gas.
A natural objective of a system provider is to maximize social welfare, i.e., to solve \eqref{eq:problem}. Since the true value $p_j$ is the private information of agent~$j$, the system designer has to employ a mechanism that incentivizes the agents to report their true values $p_j$.\footnote{We here make the standard assumption that the true values of the source vertex~$s_j$, the target vertex $t_j$, and the quantity of gas $q_j$ are public knowledge. This is reasonable since these values are physically measurable by the system provider so that misreporting them would be pointless for the agent. This assumption is also frequently made in the knapsack auction literature \cite{Aggarwal2006,Briest2011,Mualem2008}.} It is without loss of generality \cite{Gibbard1973,Myerson1981} to assume the following form of a direct revelation mechanism. The mechanism elicits a (potentially misrepresented) bid $p_j'$ from each agent $j$ and computes a solution $x(p') \in \{0,1\}^n$ to \eqref{eq:problem} based on these values. Further, the mechanism computes a payment $g_j$ for each agent~$j$. The utility of agent~$j$, when their true valuation is $p_j$ and the agents report $p'$, is then $p_j x_j(p') - g_j(p')$. The mechanism is strategyproof if truthtelling is a dominant strategy of each agent~$j$ in the underlying game where each agent chooses a value to report.

Myerson~\cite{Myerson1981} shows that an algorithm $\mathcal{A}$ can be turned into a strategyproof mechanism if and only if it is monotone in the following sense. Let $x(p')$ denote the feasible solution to \eqref{eq:problem} computed by $\mathcal{A}$ as a function of the reported valuations. Then $\mathcal{A}$ is monotone if for all agents~$j$ the function $x_j(p')$ is nondecreasing in $p'_j$ for all fixed values $p_i'$ with $i \neq j$. For a monotone algorithm, charging every agent~$j$ with $x_j(p') = 1$ the critical bid $\inf\, \{z \in \R_{\geq 0} : x_j(z,p'_{-j}) = 1\}$ and charging all other agents nothing yields a strategyproof mechanism. Here, $x_j(z,p_{-j}')$ denotes the binary variable $x_j$ as a function of the bid of agent~$j$, when the bids $p'_{-j}$ of the other agents are fixed.

We note that the algorithms designed in Sections~\ref{cf_approx} and \ref{sec:greedy} are unlikely to be monotone, since the partial enumeration schemes in both of them are not monotone. On the other hand, without the enumeration scheme, they do not provide a constant approximation, even when $W$ is a diagonal matrix.
However, by combining ideas from both algorithm, we derive a monotone algorithm with constant approximation guarantee.

\begin{algorithm}[tb]
$\vec y^* \assign$ solution of \eqref{eq:convex_relax}\;
\If{$\max\limits_{i\in [n]}p_i \ge \bigl(1 - \frac{\sqrt{3}}{e} \bigr) / \bigl(1 + \frac{4}{\sqrt{5}-1} \bigr)\, p^{\top} \vec y^*$}{
\textbf{return} $\vec\ind_{i^*}$ for $i^* \in \smash{\argmax\limits_{i\in [n]}\, p_i}$\;}
\Else{\textbf{return} solution of Greedy algorithm without partial enumeration.}	
\caption{Monotone greedy algorithm\label{algo:monotone_greedy}}
\end{algorithm}

\begin{theorem}
Algorithm \ref{algo:monotone_greedy} is a monotone $\alpha$-approximation algorithm for \eqref{eq:problem}, where $\alpha = \smash{\bigl(1 - \frac{\sqrt{3}}{e} \bigr) / \bigl(1 + \frac{4}{\sqrt{5}-1} \bigr) \approx 0.086}$. The corresponding critical payments can be computed in polynomial time.
\label{thm:monotone_greedy}
\end{theorem}

\begin{proof}
We first prove the approximation ratio.
For $\phi \define \frac{\sqrt{5}-1}{2}$ and $\rho \define 1 - \frac{\sqrt{3}}{e}$, we have $\alpha = \nicefrac{\rho}{1+\frac{2}{\phi}}$.
Let $p^*$ and $q^*$ be the optimal values of Problems \eqref{eq:problem} and \eqref{eq:convex_relax}, respectively.
Since \eqref{eq:convex_relax} is a relaxation of \eqref{eq:problem} we have that $q^* \ge p^*$. 
If $p_i\ge \alpha q^*$ for some $i\in [n]$ it follows that $p^{\top}\vec\ind_i \ge \alpha p^*$. 

Assume that $p_i < \alpha q^*$ for all $i\in [n]$, and
let $x$ be the solution computed by the greedy algorithm (without partial enumeration). 
Following the proof of \Cref{thm:greedy_approximation} we see that
$p^{\top}x \ge \rho p^* - p_j$
for some item $j\in [n]$. Since $p_j < \alpha q^*$, and by \Cref{cor:upper_bound_for_relax} we have
$q^* \le \frac{2}{\phi}p^*$, we obtain
\begin{align*}
p^{\top} x \ge \left(\rho - \frac{2\alpha}{\phi}\right)p^* = \alpha p^*.
\end{align*}
 
Next, we prove the monotonicity of the algorithm.
To this end, let $p$, $\hat p\in \N^n$ be two declared profit vectors such that there is $i \in [n]$ with $\hat p_i = p_i + 1$ and $\hat p_j = p_j$ for all $j \neq i$.
Let $x$ and $\hat x$ be the corresponding solutions computed by Algorithm~\ref{algo:monotone_greedy} and assume that $x_i = 1$. It is to show that $\hat x_i = 1$.
Let $q^*$ and $\hat q^*$ be the optimal values of \eqref{eq:convex_relax} with respect to $p$ and $\hat p$. 
Then $q^* \le \hat q^* \le q^* +1$. 
Let $H:= \{j\in [n] : p_j \ge \alpha q^*\}$ and $\hat H:= \{j\in [n] : \hat p_j \ge \alpha \hat q^*\}$.

First, assume that $H\neq \emptyset$. Since by assumption $x_i = 1$, it follows that $p_i \ge \alpha q^*$ and
$p_i = \max_{j\in [n]}p_j$.
Therefore, 
\[\hat p_i = p_i + 1 \ge \alpha q^* + 1 \ge \alpha (q^* +1) \ge \alpha \hat q^*, \]
and thus $\hat H \neq \emptyset$. Furthermore, $i$ is the only item in $\argmax_{j\in [n]}\hat p_j$ and hence $\hat x_i = 1$.

Next, assume that $H = \emptyset$. Then either $\hat H = \{i\}$ and thus $\hat x_i = 1$, or $\hat H = \emptyset$. 
In the latter case, Algorithm \ref{algo:monotone_greedy} executes the greedy algorithm for both $p$ and $\hat p$. 
But since for every $S\subseteq [n] \setminus \{i\}$
\[ \frac{\hat p_i}{w(S\cup \{i\})-w(S)} > \frac{p_i}{w(S\cup\{i\})-w(S)} \] 
and for every $j\in [n]\setminus \{i\}$ and every $S\subseteq [n] \setminus \{j\}$
\[ \frac{\hat p_j}{w(S\cup \{j\})-w(S)} = \frac{p_j}{w(S\cup\{j\})-w(S)}, \] 
when the greedy algorithm adds item $i$ to its solution after $k$ iterations for $p$, then
it also adds $i$ to its solution after at most $k$ iterations for $\hat p$.

The critical payments can be computed with binary search.
\qed\end{proof}

\section{Constantly Many Packing Constraints}
\label{sec:randomized}

In this section we generalize Problem \eqref{eq:problem} by allowing a constant number of convex quadratic constraints
and derive a constant-factor approximation algorithm using randomized rounding combined with partial enumeration. 
To this end, let $r\in \N$ be a constant natural number, and for every $k\in [r]$ let 
$\smash{\vec W^k = (w^k_{ij})_{i,j\in [n]}\in \N^{n\times n}}$ be a symmetric psd matrix with non-negative entries.
Furthermore, let $\smash{\vec p \in \N^n}$ and $\smash{c^k \in \N}$, $k\in [r]$.
We consider packing problems with~$r$ convex quadratic knapsack constraints of the form
\begin{align}
\begin{split}
\text{maximize} \qquad & \vec p^{\top} \vec x \\
\text{subject to} \qquad & 
\vec x^{\top} \vec W^k \vec x \le c^k \quad \text{ for all } k\in [r],\\
&\vec x \in \{0,1\}^n.
\end{split}
\label{multi_problem}\tag{$P^k$}
\end{align}

Denote by $\smash{\vec d^k}$ the vector consisting of the diagonal elements of $\smash{\vec W^k}$.
We obtain the following convex relaxation of \eqref{multi_problem},
\begin{equation}
\begin{aligned}
\text{maximize} \qquad & \vec p^\top \vec x \\
\text{subject to}  \qquad & 
\vec x^{\top} \vec W^k \vec x \le c^k & \text{ for all } k\in [r],\\
&(\vec d^k)^{\top} \vec x \le c^k & \text{ for all } k\in [r],\\
&\vec x \in [0,1]^n.
\end{aligned}
\tag{$R^k$}\label{multi_relax}
\end{equation}

For $\eps > 0$ we call a solution $y$ of \eqref{multi_relax} $\eps$-optimal if $p^{\top}y \ge (1-\eps)q^*$,
where~$q^*$ is the optimal value of \eqref{multi_relax}. Convex
problems of type \eqref{multi_relax} can be solved $\eps$-optimally in polynomial time by interior points methods \cite{nesterov1994interior}.

\begin{lemma}
For every $\eps > 0$ the relaxation \eqref{multi_relax} can be solved $\eps$-optimally in polynomial time. 
\end{lemma}

\begin{algorithm}[tb]
$\vec y \assign$ $\eps$-optimal solution of \eqref{multi_relax}\;
\Repeat{$x$ is feasible for \eqref{multi_problem}}{
$x \assign $ realization of $\ber(\alpha y)$
}
\textbf{return} $x$\;
\caption{Randomized rounding\label{algo:randomized_rounding}}
\end{algorithm}

We proceed to derive an approximation algorithm based on solving \eqref{multi_relax}. For some fixed value $\delta \in (0,1)$, we call items~$i$ with $w_{ii}^k \leq \delta c^k$ for all $k \in [r]$ \emph{$\delta$-light}. All other items are called \emph{$\delta$-heavy}. 
We first assume that all items are $\delta$-light and devise a randomized constant-factor approximation algorithm for 
\eqref{multi_problem} based on randomized rounding; see Algorithm \ref{algo:randomized_rounding}. 
To that end, for some vector \mbox{$y\in [0,1]^n$}, denote by $\ber(y)$ the vector of stochastically independent binary random variables
$\vec X = (X_1,\dots, X_n)^{\top}$ with the property $\Pr[X_i = 1] = y_i$ and $\Pr[X_i = 0] = 1-y_i$, for $i\in [n]$. 

\begin{lemma}\label{lem:randomized_rounding}
Let $\delta \in (0,1)$ and assume that all items $i \in [n]$ are $\delta$-light. 
Let $\eps\in (0,1)$, $p^*$ be the optimal value of \eqref{multi_problem}, $y$ be an $\eps$-optimal solution of \eqref{multi_relax}, $\alpha\in (0,1)$, and
$\vec X = \ber(\alpha y)$.
Then, $\E[p^{\top}X \mid X \text{ is feasible}] \ge f(\alpha,\delta) (1-\eps)p^*$, 
where 
$f(\alpha,\delta) = \alpha\big(1 - g(\alpha,\delta)\big)^r$ and $g(\alpha,\delta) = \alpha\bigl( 1 + (1+ \delta^{\frac{1}{3}})^3 \bigr) + (1-\alpha)\delta$.
\end{lemma}

\begin{proof}
By Bayes' theorem we have
\begingroup\allowdisplaybreaks
\begin{align}
\E[\vec p^\top \vec X \mid \vec X \text{ feasible}] & = \sum_{\ell\in [n]}p_\ell\, \E[X_\ell\mid \vec X \text{ feasible}] \nonumber\\
&= \sum_{\ell\in [n]}p_\ell\, \Pr[X_\ell = 1 \mid \vec X \text{ feasible}] \nonumber\\
&= \sum_{\ell\in [n]}p_\ell\, \frac{ \Pr[X_\ell =1]\, \Pr[\vec X \text{ feasible} \mid X_\ell=1]}{\Pr[\vec X \text{ feasible}]} \nonumber\\
&\ge \alpha\sum_{\ell\in [n]}p_\ell\, y_\ell\,\Pr[\vec X^\top \vec W^k \vec X \le c^k \text{ for all }k\in [r] \mid X_\ell = 1] \nonumber\\
&\ge \alpha\sum_{\ell\in [n]}p_\ell\, y_\ell\, \prod_{k=1}^r\Pr[\vec X^\top \vec W^k \vec X \le c^k\mid X_\ell = 1], \label{expected_obj}
\end{align}%
\endgroup
where the last inequality follows from the monotonicity of the functions $\vec X \mapsto \vec X^\top \vec W^k \vec X$  
with respect to the natural partial order on $\R^n$ 
and the FKG inequality; see \cite{fortuin1971correlation}.
In the following, we show that for every $\ell\in [n]$ and $k\in [r]$ we have
\begin{equation}\label{pr_w_le_c}
\Pr[\vec X^\top \vec W^k \vec X \le c^k  \mid X_{\ell} = 1] \ge 1-g(\alpha,\delta).
\end{equation}
Combining \eqref{expected_obj} and \eqref{pr_w_le_c}, and using that $p^{\top}y \ge (1-\eps)p^*$ we then obtain
\[
\E[\vec p^\top \vec X \mid \vec X \text{ feasible\,}] \ge \alpha\, p^{\top}y\, (1-g(\alpha,\delta))^r = f(\alpha,\delta)\, \vec p^{\top}{\vec y} 
\ge f(\alpha, \delta) (1-\eps)\vec p^*,
\]
and we are finished.

We proceed as follows. Let $\ell\in [n]$ and $k\in [r]$.
Define ${\vec z} \in \R^n$ by $ z_i = y_i$ for all $i\in [n]\setminus \{\ell\}$ and $ z_{\ell} = 1$.
We derive an upper bound on ${\vec z}^\top \vec W^k {\vec z}$ and then use this bound to prove
that $\E[\vec X^\top  \vec W^k \vec X \mid X_{\ell} = 1] \le g(\alpha,\delta)\,c^k$. Using Markov's inequality yields \eqref{pr_w_le_c}. 
For the sake of readability, throughout the rest of the proof we omit the superscripts and simply write $\vec W = \vec W^k$ and $c = c^k$. 

\begin{claim}
We have
\begin{align}
{\vec z}^\top \vec W {\vec z} \le \min_{\gamma\in (0,1)}\left[\frac{\delta}{(1-\gamma)^2} + \frac{1}{\gamma^2}\right] c = \bigl(1+\delta^{\frac{1}{3}}\bigr)^3c. \label{upper_bound}
\end{align}
\end{claim}

\begin{proof}[of the claim]
Let $\vec W = \vec U^{\top} \vec U$, with $\vec U = (u_{ij})_{i,j\in [n]}$, be the Cholesky decomposition of $\vec W$ 
and denote by $\vec u_i\in \R^n$, $i\in [n]$, the rows of $\vec U$. It follows that
\begin{align*}
{\vec z}^{\top} \vec W {\vec z} = {\vec z}^{\top} \vec U^{\top} \vec U {\vec z} = \norm{\vec U {\vec z}}^{2} = \sum_{i\in [n]} (\vec u_i^{\top} {\vec z})^2.
\end{align*}

Let $\gamma\in (0,1)$ and $i\in [n]$. There are two possible cases. 
Either $\abs{\vec u_i^{\top}{\vec y}} \ge \gamma \abs{\vec u_i^{\top}{\vec z}}$ and thus $\abs{\vec u_i^{\top} {\vec z}} \le \frac{1}{\gamma} \abs{\vec u_i^{\top}{\vec y}}$.
Or we have $\abs{\vec u_i^{\top}{\vec y}} < \gamma \abs{\vec u_i^{\top}{\vec z}}$. 
But then, 
\begin{align*}
\abs{u_i^{\top} z} = \abs{u_i^{\top} y + (1 -  y_{\ell})u_{i\ell}}
\le \abs{u_i^{\top} y} + (1 -  y_{\ell})\abs{u_{i\ell}}
< \gamma \abs{u_i^{\top} z} + \abs{u_{i\ell}}.
\end{align*}
Hence, 
\[ \abs{u_i^{\top} z} < \frac{1}{1 - \gamma} \abs{u_{i\ell}} = \frac{1}{1 - \gamma} \abs{u_i^{\top}\vec\ind_{\ell}}. \] 
It follows that in any of the two cases we have 
\begin{align*}
(u_i^{\top} z)^2 \le \frac{1}{(1-\gamma)^2}(u_i^{\top}\vec\ind_{\ell})^2 + \frac{1}{\gamma^2}(u_i^{\top} y)^2.
\end{align*}
Using that $w_{\ell\ell} \le \delta c$, we conclude that
\begin{align*}
 z^{\top}W z &= \sum_{i\in [n]}(u_i^{\top} z)^2 
\le \frac{1}{(1-\gamma)^2}\sum_{i\in [n]}(u_i^{\top}\vec\ind_{\ell})^2 + \frac{1}{\gamma^2}\sum_{i\in [n]}(u_i^{\top} y)^2 \\
&= \frac{\vec\ind_{\ell}^{\top} W\vec\ind_{\ell}}{(1-\gamma)^2} + \frac{\vec y^{\top}W\vec y}{\gamma^2}
\le \frac{\delta c}{(1-\gamma)^2} + \frac{c}{\gamma^2} \\
&= \left[\frac{\delta}{(1-\gamma)^2} + \frac{1}{\gamma^2}\right] c.
\end{align*}
Applying standard calculus we see that the function $\smash{\gamma\mapsto\frac{\delta}{(1-\gamma)^2} + \frac{1}{\gamma^2}}$ attains its minimal value $(1+\delta^{\frac{1}{3}})^3$ at
$\gamma = ( 1+\delta^{\frac{1}{3}})^{-1}$.
This completes the proof of the claim.
\qed \end{proof}

Using the bound \eqref{upper_bound} we proceed to derive an upper bound on the expected value of $\vec X^{\top}W \vec X$ conditioned on $X_{\ell} = 1$. 
To that end, let $N_{\ell}\define [n]\setminus \{\ell\}$.
We have
\begin{align}
\E[&\vec X^{\top} W \vec X \mid X_{\ell} = 1] \nonumber\\
&= \E\left[ \sum_{i, j\in N_{\ell}}w_{ij}X_iX_j + 2\sum_{i\in N_{\ell}}w_{i\ell}X_i + w_{\ell\ell} \right] \nonumber\\
&= \sum_{i, j\in N_{\ell}: \atop i\neq j}w_{ij}\E[X_i]\E[X_j] + \sum_{i\in N_{\ell}}w_{ii}\E[X_i^2] + 2\sum_{i\in N_{\ell}}w_{i\ell}\E[X_i] + w_{\ell\ell} \nonumber\\
&= \alpha^2 \sum_{i, j\in N_{\ell}: \atop i\neq j } w_{ij}y_iy_j 
+ \alpha\sum_{i\in N_{\ell}}w_{ii}y_i + 2\alpha \sum_{i\in N_{\ell}}w_{i\ell}y_i + \alpha w_{\ell\ell}  + (1-\alpha)w_{\ell\ell}, \nonumber\\
&\le \alpha\Big(\sum_{i, j\in N_{\ell}: \atop i\neq j } w_{ij}y_iy_j + 2\sum_{i\in N_{\ell}}w_{i\ell}y_i + w_{\ell\ell}\Big) + 
\alpha\sum_{i\in N_{\ell}}w_{ii}y_i  + (1-\alpha)w_{\ell\ell}, \nonumber \\
&= \alpha \vec z^{\top} \vec W \vec z  + \alpha \sum_{i\in N_{\ell}}w_{ii}y_i + (1-\alpha)w_{\ell\ell}, \label{upperbound_Ew}
\end{align}
where the inequality follows from $\alpha\in (0,1)$.
Since $y$ is a feasible solution of \eqref{multi_relax}, we have
\begin{align}
\sum_{i\in N_{\ell}}w_{ii}y_i \le c. \label{feas_sol_bound}
\end{align}
Plugging \eqref{upper_bound}, \eqref{feas_sol_bound}, and $w_{\ell\ell} \le \delta c$ into \eqref{upperbound_Ew} yields
\begin{align*}
\E[\vec X^{\top}\vec W \vec X \mid X_{\ell} = 1] \le \alpha (1+\delta^{\frac{1}{3}})^3c + \alpha c + (1-\alpha)\delta c = g(\alpha, \delta) c.
\end{align*}
Therefore, by Markov's inequality,
\begin{align*}
\Pr[\vec X^{\top}\vec W \vec X \le c &\mid X_{\ell} = 1] \\
&= 1 - \Pr[\vec X^{\top}\vec W\vec X  > c \mid X_{\ell} = 1] \nonumber \\
&\ge 1 - \Pr\left[\vec X^{\top}\vec W\vec X  \ge \frac{\E[\vec X^{\top}\vec W\vec X \mid X_{\ell} =1]}{g(\alpha, \delta)}
 \mid X_{\ell} = 1\right] \nonumber \\
&\ge 1 - g(\alpha, \delta).
\end{align*}
This establishes Inequality \eqref{pr_w_le_c} and completes the proof.
\qed\end{proof}

In order to maximize the approximation guarantee of Algorithm \ref{algo:randomized_rounding}, we need to find
$\alpha\in (0,1)$ that maximizes $f(\alpha,\delta)$. 

\begin{lemma}\label{lem:best_alpha}
For every $\delta\in (0,1)$ the function $(0,1)\to \R$, $\alpha\mapsto f(\alpha, \delta)$ attains its maximum at 
\begin{align*}
\alpha_{\delta} \define \frac{1 - \delta}{(r+1)\left[ 1 - \delta + (1 + \delta^{\frac{1}{3}})^3 \right]},
\end{align*}
and it holds that
\[\lim_{\delta \to 0} f(\alpha_{\delta},\delta) = \frac{1}{2(r+1)}\left( \frac{r}{r+1} \right)^r \ge \frac{1}{2e(r+1)}. \]
\end{lemma}

\begin{proof}
The fact that the function $(0,1)\to \R$, $\alpha\mapsto f(\alpha,\delta)$ attains its maximum value 
at 
$\alpha_{\delta}$
can be verified using standard calculus.
By the continuity of $f$ and  $\alpha_{\delta}$ it follows that
\[ \lim_{\delta\to 0}f(\alpha_{\delta},\delta) = f(\alpha_0,0) = \frac{1}{2(r+1)}\left( \frac{r}{r+1} \right)^r \ge \frac{1}{2e(r+1)},   \] 
which completes the proof.
\qed\end{proof}
We proceed to show that for this $\alpha$, the probability that the random vector $X = \ber(\alpha y)$ produced by Algorithm \ref{algo:randomized_rounding}
is infeasible for \eqref{multi_problem} can be bounded from above by $\tfrac{1}{2}$.

\begin{lemma}
Let $y$ be an optimal solution to \eqref{multi_relax}, $\alpha\in (0,1)$, and $\vec X = \ber(\alpha y)$.
Then $\Pr[X \text{ infeasible for \eqref{multi_problem}}] \le r(\alpha^2 + \alpha)$. 
In particular, if $\alpha = \alpha_{\delta}$, then
$\Pr[X \text{ infeasible for \eqref{multi_problem}}] \le \frac{1}{2}$.
\end{lemma}

\begin{proof}
Since for every $i$, $j\in [n]$ with $i\neq j$, $X_i$ and $X_j$ are stochastically independent, it holds for every $k\in [r]$ that
\begin{align*}
\E[\vec X^\top \vec W^k \vec X] &= \sum_{i,j\in [n] : i\neq j} w^k_{ij}\E[X_i]\E[X_j] + \sum_{i\in [n]}w^k_{ii}\E[X_i^2] \\
&= \alpha^2\sum_{i,j\in [n] : i\neq j}w^k_{ij}y_iy_j  + \alpha\sum_{i\in [n]}w^k_{ii}y_i \\
&\le (\alpha^2 + \alpha)c^k,
\end{align*}
where the inequality follows from the fact that $\vec y$ is a feasible solution of \eqref{multi_relax}.
Thus, Markov's inequality implies
\begin{align*}
\Pr[\vec X \text{ not feasible\,}] 
&= \Pr[\vec X^{\top}\vec W^k \vec X > c^k\; \text{for some } k\in [r]] \\
&\le \sum_{k=1}^r \Pr[\vec X^{\top}\vec W^k \vec X > c^k]  \\
&\le \sum_{k=1}^r \Pr\biggl[\vec X^{\top}\vec W^k \vec X \ge \frac{\E[\vec X^{\top}\vec W^k \vec X]}{\alpha^2+\alpha}\biggr] \\
&\le r (\alpha^2 + \alpha).
\end{align*}

Finally, for every $\delta\in (0,1)$
\begin{align*}
r[\alpha_{\delta}^2 + \alpha_{\delta}] &\le r[\alpha_0^2 + \alpha_0] \\
&= r\left[ \frac{1}{4(r+1)^2} + \frac{1}{2(r+1)} \right] \le \frac{1}{2},
\end{align*}
as required.
\qed
\end{proof}

To finish the proof, we show that for any constant $\delta \in (0,1)$, any optimal solution to \eqref{multi_problem} contains a constant number of $\delta$-heavy items only.

\begin{lemma}
Let $x^*$ be an optimal solution to problem \eqref{multi_problem}, let $\delta\in (0,1)$, and let
$H^*\define \{i\in [n]: x_i^* = 1 \text{ and } i \text{ is } \delta\text{-heavy} \}$. 
Then $|H^*| \le \frac{r}{\delta}$.
\label{lem:heavy_items_bounded}
\end{lemma}

\begin{algorithm}[tb]
$H_{\delta} \assign \{i\in [n]: \exists k\in [r] \text{ with } w_{ii}^k > \delta c^k\}$ \;
$z_{\delta} \assign$ optimal solution of \eqref{multi_problem} with $x_i = 0\,\forall i\in [n]\setminus H_{\delta}$ and $|N_1(x)| \le \frac{r}{\delta}$ (via enumeration) \;
$y_{\delta} \assign$ approximate solution of \eqref{multi_problem} with $x_i = 0 \,\forall i\in H_{\delta}$ 
computed by randomized rounding (Algorithm \ref{algo:randomized_rounding}) with $\alpha = \alpha_{\delta}$\;
\textbf{return} $\argmax_{x\in \{y_{\delta},z_{\delta}\}} p^{\top} x$
\caption{Randomized rounding with partial enumeration \label{algo:randomized_rounding_enumeration}}
\end{algorithm}

\begin{proof}
Let $N_1(x^*)\define \{i\in [n]: x^*_i = 1\}$ and
$H\define \{i\in [n] : i \text{ is } \delta\text{-heavy}\}$.
Then $H^* = N_1(x^*) \cap H$. Furthermore, we have 
$H = \bigcup_{k=1}^{r}H_k$, where
$H_k\define \{i\in [n] : w_{ii}^k > \delta c^k\}$. 
Since $\vec x^*$ is feasible for \eqref{multi_problem}, for every $k\in [r]$ we have 
\begin{align*}
c^k \ge (\vec x^*)^{\top}\vec W^{k} \vec x^* \ge \sum_{i\in N_1(\vec x^*)} w^k_{ii} 
\ge \card{N_1(\vec x^*) \cap H_k} \delta c^k,
\end{align*}
and thus $\card{N_1(\vec x^*) \cap H_k} \le \frac{1}{\delta}$.
It follows that 
\begin{align*}
|H^*| = \Bigl| N_1(\vec x^*) \cap \bigcup_{k=1}^{r}H_k \Bigr| \le \sum_{k=1}^r |N_1(\vec x^*)\cap H_k| \le \sum_{k=1}^r \frac{1}{\delta} =  \frac{r}{\delta},
\end{align*}
which completes the proof.
\qed\end{proof}

We are now in position to devise a randomized constant-factor approximation algorithm for Problem \eqref{multi_problem}. The algorithm first enumerates all solutions using only heavy items, then computes a solution with randomized rounding involving only the light items, and returns the better of the two solutions; see Algorithm~\ref{algo:randomized_rounding_enumeration}.

\begin{theorem}
\label{thm:rand_rounding_ratio}
For every $\bar{\eps} > 0$, there are $\eps > 0$ and $\delta > 0$ such that Algorithm~\ref{algo:randomized_rounding_enumeration} yields an $(\alpha+\bar{\eps})$-approximation for \eqref{multi_problem} where 
\[\alpha = \frac{1}{1 + 2(r+1)(\frac{r+1}{r})^r} \geq \frac{1}{1 + 2e(r+1)}.\]
\end{theorem}

\begin{proof}
Let $\epsilon > 0$ and $\delta > 0$ be arbitrary. We claim that Algorithm~\ref{algo:randomized_rounding_enumeration} yields a $\rho_{\eps,\delta}$-approximation where
\[ \rho_{\eps,\delta} \define \frac{f_{\eps,\delta}}{1 + f_{\eps,\delta}}, \] 
with $f_{\eps,\delta} \define f(\alpha_{\delta},\delta)(1-\eps)$.

Let $x^*$ be an optimal solution of \eqref{multi_problem}.
We distinguish two cases.

\noindent
\textbf{First case: $\sum_{i\in H_{\delta}} p_i x_i^* \ge \rho_{\eps,\delta} p^{\top}x^*$.}
Then, by \Cref{lem:heavy_items_bounded}, 
\[
p^{\top}z_{\delta} \ge \sum_{i\in H_{\delta}}p_ix^*_i \ge \rho_{\eps,\delta} p^{\top}x^*.
\] 

\noindent
\textbf{Second case: $\sum_{i\in H_{\delta}}p_i x^*_i < \rho_{\eps,\delta} p^{\top}x^*$.} Thus,
$\sum_{i\in [n]\setminus H_{\delta}}p_i x^*_i \ge (1-\rho_{\eps,\delta}) p^{\top}x^*$. 
\Cref{lem:randomized_rounding} yields
\begin{align*}
p^{\top}y_{\delta} \ge f_{\eps,\delta} \sum_{i\in [n]\setminus H_{\delta}} p_i^{\top}x^*_i
\ge f_{\eps,\delta}(1 - \rho_{\eps,\delta}) p^{\top} x^*
= \rho_{\eps,\delta} p^{\top}x^*.
\end{align*}

Finally, by the continuity of $f$ and $\alpha_{\delta}$ we obtain
\begin{align*}
\lim_{\eps,\delta\to 0}\rho_{\eps,\delta} = \frac{f(\alpha_0,0)}{1 + f(\alpha_0,0)} = \frac{1}{1 + 2(r+1)\left( \frac{r+1}{r} \right)^r},\end{align*}
which completes the proof.
\qed\end{proof}

\section{Approximation Hardness}
\label{sec:hardness}

In this section, we show that packing problems with convex quadratic constraints of type \eqref{eq:problem} are $\mathsf{APX}$-hard.

\begin{theorem}\label{no_ptas3}
It is $\NP$-hard to approximate packing problems with convex qua\-dra\-tic constraints by a factor of $\frac{91}{92} + \eps$, for any $\eps > 0$.
\end{theorem}

\begin{proof}
We reduce from the $6$-set packing problem which is $\NP$-hard to approximate by a factor of $\frac{22}{23} + \eps$  for all $\eps > 0$; see Hazan et al.~\cite{hazan2003}. 
An instance of a $6$-set packing is given by a ground set $[m]$ and a family $\mathcal{S} \subseteq 2^{[m]}$ of subsets of $[m]$ such that $|S| = 6$ 
for all $S \in \mathcal{S}$. A subfamily $\mathcal{S}^* \subseteq \mathcal{S}$ is a feasible solution to the $6$-set packing problem 
if $S \cap T = \emptyset$ for all $S,T \in \mathcal{S}^*$.
For a given instance of $6$-set packing, and a value $k \in \N$ the \emph{gap problem} is the decision problem to decide whether:
\begin{description}
	\item[$\mathsf{Yes}$:] there is a solution to the $6$-set packing problem of size at least $k$, or
	\item[$\mathsf{No}$:] every solution has size strictly smaller than $\frac{22}{23}k$.
\end{description}
For optimal sizes in the interval $[\frac{22}{23}k,k)$ any answer is admissible. The approximation hardness of $6$-set packing implies that the gap problem is an $\NP$-hard decision problem. 

Let $n \define |\mathcal{S}|$ and number the sets $\mathcal{S} = \{S_1,S_2,\dots,S_n\}$. 
Let $\vec A =  (a_{ij})_{i,j}\in \{0,1\}^{m \times n}$ be defined as  $a_{ij} = 1$ if and only if $i \in S_j$, and let $\vec W = \vec A^{\top} \vec A$. Consider the problem
  \begin{align}
   \begin{split}
  \text{maximize}\qquad  & \ones^{\top} \vec x \\
      \text{subject to}\qquad & \vec x^{\top} \vec W \vec x \le 6k,\\
      &\vec x \in \{0,1\}^n,
    \end{split}
    \tag{$SP$}
    \label{eq:6SPP}
  \end{align}
  where $\ones = (1,\dots,1)^\top$ is the all-ones vector. We calculate
  \begin{equation}\label{eq:6SPP_weight}
  \vec x^{\top} \vec W \vec x = \norm{\vec A \vec x}_2^2 =
    \sum_{i=1}^m \bigg(\sum_{j=1}^n a_{ij}\, x_j\bigg)^{\!2} =
    \sum_{i=1}^m \Bigl( \sum_{j \in [n] : i \in S_j} x_{j}\Bigr)^2.
  \end{equation}

Suppose, we have a $\mathsf{Yes}$-instance for the gap problem and let $\mathcal{S}^*$ be a subset of pairwise disjoint sets of cardinality $k$. 
Then a feasible solution for \eqref{eq:6SPP} is given by $\vec x^*$ defined as $x^*_j = 1$ if $S_j \in \mathcal{S}^*$, and $x^*_j = 0$, otherwise. 
Since every set $S_j$, $j \in [n]$, contributes at least $6$ to the left hand side of the knapsack constraint \eqref{eq:6SPP_weight}, 
this solution is also optimal for \eqref{eq:6SPP} and has an objective value of $k$.

Next, consider a $\mathsf{No}$-instance for the gap problem, let $\vec x^*$ be a corresponding optimal solution of \eqref{eq:6SPP}, and let $k'$ be its objective value. 
Since for a $\mathsf{No}$-instance every solution of the $6$-set packing problem has size strictly less than $\frac{22}{23}k$, 
every set that is picked beyond the first $\lfloor \frac{22}{23}k \rfloor$ sets, 
intersects at least once with at least one of the first $\lfloor \frac{22}{23}k \rfloor$ sets. 
Thus, the first $\lfloor \frac{22}{23}k \rfloor$ sets each contribute at least $6$ to \eqref{eq:6SPP_weight}, 
and each of the further $k' - \lfloor \frac{22}{23}k \rfloor$ sets each contributes at least $5 + 4 -1 = 8$ to \eqref{eq:6SPP_weight}. 
We obtain
\begin{align*}
6k &\geq (\vec x^*)^{\top} \vec W \vec x^*\\
&\ge 6 \bigg\lfloor\frac{22}{23}k\bigg\rfloor + 8\biggl(k' - \bigg\lfloor\frac{22}{23}k\bigg\rfloor\biggr)\\
&\geq 8k' - \frac{44}{23}k
\end{align*}
implying $k' \leq \frac{91}{92}k$.
We conclude that for a $\mathsf{Yes}$-instance the objective value of \eqref{eq:6SPP} is at least $k$ while for a $\mathsf{No}$-instance it is strictly less than $\frac{91}{92}k$. Therefore, the problem is $\NP$-hard to approximate by a factor of $\frac{91}{92} + \eps$ for any $\eps > 0$.
\qed\end{proof}

\section{Computational results}
\label{sec:computational_results}

We apply our algorithms to a problem of the type described in \Cref{ex:gas}.
Specifically, we solve the welfare maximization problem for instance~134 of the 
GasLib library \cite{SABHJKKIOSSS17}; see \Cref{fig:greece} for an illustration of the network $G=(V,E)$.
\begin{figure}[tb]
\centering
\input{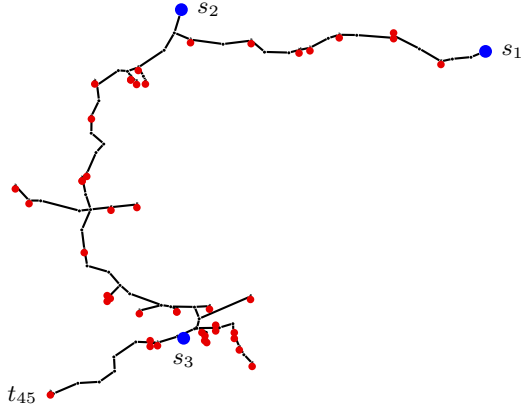}
\caption{The Gaslib-134 instance. Sources are shown in blue, sinks in red.}
\label{fig:greece}
\end{figure}
The instance contains upper and lower pressure bounds for every node $v\in V$ as well as
all physical properties to compute the pipe resistances $\beta_e$, $e\in E$.

Sources and sinks are denoted by~$S$ and $T$, respectively.
Every sink $t\in T$ requests a transportation of $q_t$ units of gas to $t$. 
To ensure the robustness of the network in the sense of \cite{Klimm2019}, 
we assume that all sinks between $s_1$ and $s_2$ are (possibly) supplied by $s_1$,
all sinks between $s_3$ and $t_{45}$ by $s_3$, and all other sinks by $s_2$. 
Denote the set of all sinks that are (possibly) supplied by $s_i$ by $T_i$, 
$i=1,2,3$.
For simplicity, we assume that the economic welfare is proportional to the amount of transported gas. 
That is, there is a constant $\theta > 0$ such that 
for every sink $t\in T$ the economic welfare $p_t$ of transporting $q_t$ units of gas
to $t$ equals $\theta q_t$. 

Our goal is to choose a welfare-maximal subset of transportations
that can be satisfied simultaneously while the pressures 
at the first sink $s_1$ and the last source $t_{45}$ are within their feasible interval. 
To that end, let $\bar E$ denote the path from $s_1$ to $t_{45}$, and for
every $t\in T_i$ denote by $E_t$ the set of edges on the unique path from $s_i$ to $t$, $i = 1,2,3$.
Let $p = (p_t)_{t\in T}$, $W = (w_{t,t'})_{t,t'\in T}$,
with $w_{t,t'} = \sum_{e\in \bar E \cap E_t\cap E_{t'}}\beta_e q_tq_{t'}$, 
and let $c = \bar \pi_{s_1} - \ubar \pi_{t_{45}}$,
where for a node $v\in V$, $\bar\pi_v$ and $\ubar\pi_v$ denote the upper and lower bound on the squared pressure at $v$, respectively.
Finally, let $x = (x_t)_{t\in T} \in \{0,1\}^T$, where $x_t = 1$ if and only if sink $t$ is supplied.
Then, the welfare-maximization problem can be formulated as \eqref{eq:problem}; see \Cref{ex:gas}.

The GasLib-134 instance contains 1234 different scenarios, where for each scenario demands $\hat q_t$ are given for every sink $t\in T$.
In order to make the optimization problem non-trivial we increase the node demands by setting $q_t = \gamma\hat q_t$,
for $\gamma\in \Gamma\define\{5,10,50,100\}$.
We apply the golden ratio algorithm, the greedy algorithm, and randomized rounding to the first 100 scenarios. 
For each scenario we consider every $\gamma\in \Gamma$.
Each of the three algorithms is executed in three different versions, one without partial enumeration, one with partial enumeration with
one initial item, and one with partial enumeration with two initial items.

We run randomized rounding with $\alpha$ chosen uniformly at random from $[0,1]$. 
Instead of a single feasible realization, we generate 100 feasible realizations of $\ber(\phi y)$
and return the one with the highest profit.
With the golden ratio algorithm, instead of scaling an optimal solution $y$ of \eqref{eq:convex_relax} by $\phi$ in order to obtain a 
feasible solution of \eqref{eq:nonconvex_relax}, we scale it by the largest number $\lambda\in [\phi,1]$ such that $\lambda y$ is
feasible for \eqref{eq:nonconvex_relax}. We find this number $\lambda$ using binary search. 
In addition, we use the improvements described in \Cref{rem:swapping_enhancement}.

The result of each algorithm is compared to an optimal solution computed with a standard MIP solver applied to the following MIP
\begin{align*}
\text{maximize} \qquad &\vec p^\top \vec x\\
\text{subject to} \qquad &\sum_{i=1}^n z_i \leq c,\\
& z_i \ge \sum_{j=1}^n w_{ij}(x_i+x_j-1) \quad \text{for all } i\in [n], \\
& z_i \ge 0 \quad \text{for all } i\in [n],\\
&\vec x \in \{0,1\}^n, \label{eq:problem}
\end{align*}
which can be shown to be equivalent to \eqref{eq:problem}.

\begin{figure}[tb]
  \centering
  \includegraphics[width=0.754\textwidth]{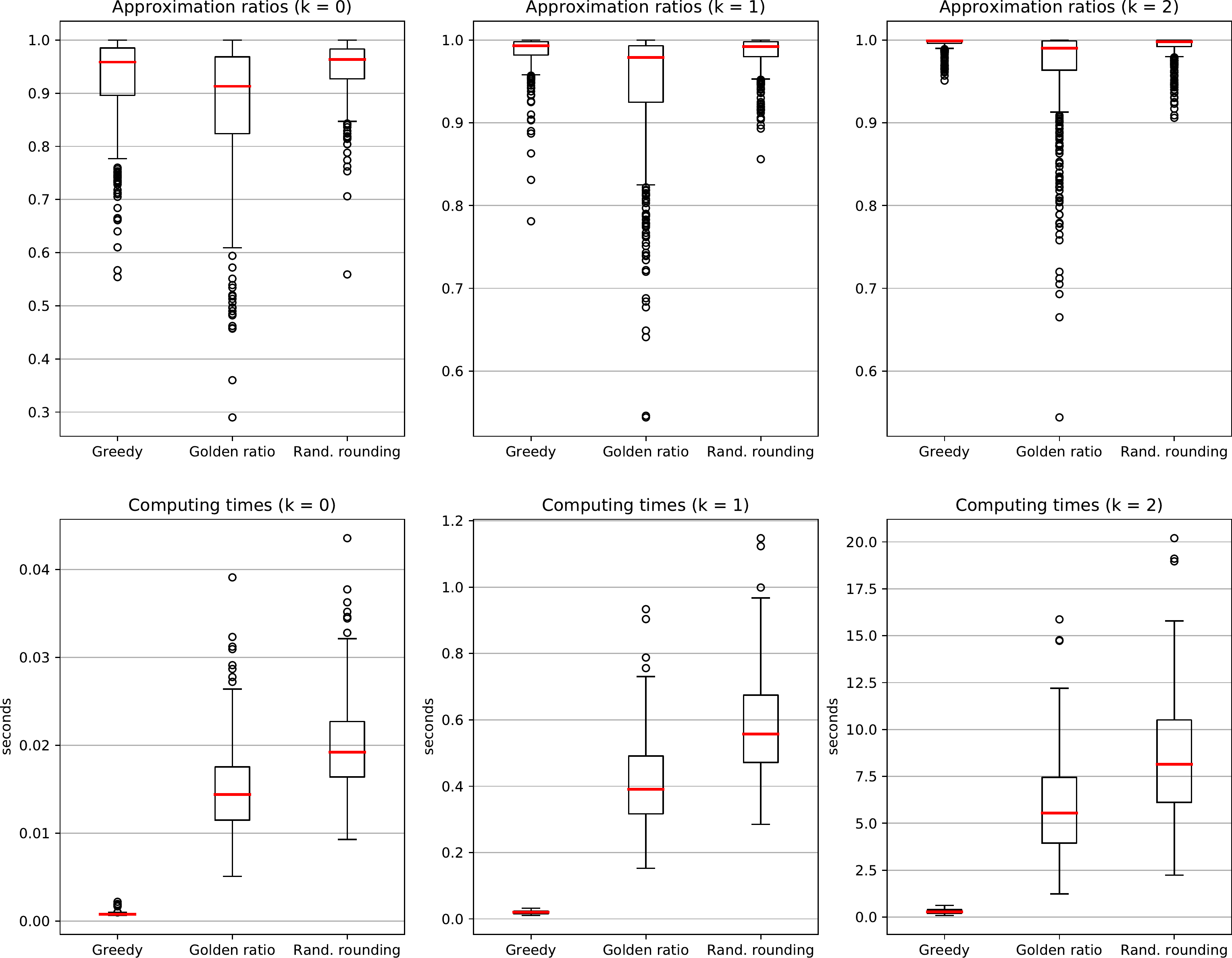}
  \caption{Approximation ratios (top row) and computation times (bottom row) of the three algorithms when executed with partial enumeration of $k = 0, 1, 2$ elements. The red line indicates the median.}
  \label{fig:computational_results}
\end{figure}

The computations are executed on a 6-core AMD Phenom II X6 1090T processor with 
3.3 GHz. The code is implemented in Python 3.6 and we use the SLSQP algorithm of the SciPy \texttt{optimize} package 
to solve the convex relaxation \eqref{eq:convex_relax}.
The results are shown as box plots in \Cref{fig:computational_results} and in \Cref{tab:computational_results}.

\begin{table}[tb]
\centering
\begin{tabularx}{\textwidth}{@{}l@{\extracolsep{\fill}}cc cc cc@{}}
\toprule
& \multicolumn{2}{c}{No enumeration} & \multicolumn{2}{c}{1-enumeration} & \multicolumn{2}{c}{2-enumeration} \\
& Mean & SD & Mean & SD & Mean & SD \\
\midrule
Greedy & 0.925 & 0.0837 & \textbf{0.985} & 0.0228& \textbf{0.996} &  0.0079\\
Golden ratio & 0.875 & 0.1288 &0.944  & 0.0773 & 0.962 &  0.0639\\
Rand.\ rounding & \textbf{0.948} & 0.0504 &0.984  & 0.0220& 0.991 & 0.0160\\
\bottomrule
\end{tabularx}
\caption{Mean and standard deviation (SD) of the approximation ratio 
of the greedy algorithm, the golden ratio algorithm, and randomized rounding.
Each algorithm has been executed without partial enumeration (left), with partial enumeration with one initial item (middle), and
with partial enumeration with two initial items (right). \label{tab:computational_results}}
\end{table}

We observe that the greedy algorithm on average achieves the best approximation ratios when combined with partial enumeration. 
At the same time it runs approximately 20 times faster than the golden ratio algorithm and randomized rounding.
The slower running time of these two algorithms is due to the fact that they rely on solving the convex relaxation first. 
The approximation ratio of all three algorithms is on average much higher than their proven worst case lower bounds.
However, the quality of the solutions produced by the golden ratio algorithm is subject to strong fluctuations.
By running the algorithm with partial enumeration with three initial items we could guarantee 
a ratio of at least $\phi$ for every instance, as was shown in \Cref{thm:phi_approximation}.

\bibliography{literature}

\end{document}